\documentclass[12pt]{article}

\usepackage[utf8]{inputenc}
\usepackage[top=1in,bottom=1in,right=1in,left=1in]{geometry}
\usepackage{amssymb}
\usepackage{amsmath}
\usepackage{mathtools}
\usepackage{graphicx}
\usepackage{bm}
\usepackage{mathrsfs}
\usepackage{mathrsfs}
\usepackage{outlines}
\usepackage{enumitem}
\usepackage{hyperref}
\hypersetup{
  colorlinks   = true,    % Colours links instead of ugly boxes
  urlcolor     = blue,    % Colour for external hyperlinks
  linkcolor    = blue,    % Colour of internal links
  citecolor    = cyan      % Colour of citations
}

\usepackage[flushleft]{threeparttable}
\usepackage{minted}
\usepackage{float}
\usepackage{minted}
\usepackage[authoryear]{natbib}
\setcitestyle{semicolon,aysep={, },yysep={;}}
\usepackage{mwe}
\usepackage{booktabs,caption}
\usepackage{babel}
\usepackage{subcaption}

\usepackage{enumerate}
\usepackage{enumitem}

\usepackage{setspace}

\usepackage{comment}

\usepackage{amsthm}
\theoremstyle{plain}
\newtheorem{lemma}{Lemma}
\newtheorem*{theorem*}{Theorem}
\usepackage{dsfont}
\usepackage{bbm}

\usepackage{bbm}
\usepackage{tikz}

\usepackage{listings}
\usepackage{xcolor}
\definecolor{ForestGreen}{RGB}{34,139,34}

\DeclarePairedDelimiter\evaluat{.}{\rvert}
\reDeclarePairedDelimiterInnerWrapper\evaluat{nostarscaled}{\mathopen{}#2\mathclose{#3}}

\newcommand{\cA}{\mathcal{A}}
\newcommand{\cX}{\mathcal{X}}

\newcommand{\cP}{\mathcal{P}}
\newcommand{\cE}{\mathcal{E}}

\newcommand{\cM}{\mathcal{M}}
\newcommand{\e}{\varepsilon}

\usepackage[]{todonotes}

\definecolor{Green}{RGB}{34,139,34}

\input{ee.sty}

\renewcommand{\Pr}{\ensuremath{\mathrm{Pr}}}

\newtheorem{cor}{Corollary}
\newtheorem{prp}{Proposition}
\newtheorem{rmk}{Remark}
\newtheorem{thm}{Theorem}
\newtheorem{exm}{Example}
\newtheorem{dfn}{Definition}

\title{\textbf{\LARGE Equivalence testing with data-dependent and post-hoc equivalence margins}\footnote{We are grateful to Isaiah Andrews, Josha Dekker, Peter Gr\"unwald, Jesse Hemerik, Wouter Koolen, Wietze Koops, Diego Martinez-Taboada, Sam van Meer, Iosif Pinelis, Aaditya Ramdas and participants at the CWI E-readers group and the 13th Econometric Internal PhD Conference at Erasmus University Rotterdam for their insightful comments.}} %We also acknowledge the use of ChatGPT for proofreading and assistance in finding useful literature. Koobs conducted part of this research while visiting the Department of Economics at Boston University.}}%\footnote{This is a preliminary draft. Please do not cite or distribute without explicit permission of the authors.}}

\author{Stan Koobs\footnote{Econometric Institute, Erasmus University Rotterdam. Email: \href{mailto:koobs@ese.eur.nl}%
{koobs@ese.eur.nl}.}%
\and Nick W. Koning\footnote{Econometric Institute, Erasmus University Rotterdam. Email: \href{mailto:n.w.koning@ese.eur.nl}%
{n.w.koning@ese.eur.nl}.}}%\\

\date{\today}

\begin{document}

\maketitle

\vspace{-1.1cm}

\begin{abstract}
    Equivalence testing compares the hypothesis that an effect $\mu$ is large against the alternative that it is negligible.
    Here, `large' is classically expressed as being larger than some `equivalence margin' $\Delta$.
    A longstanding problem is that this margin must be specified but can rarely be objectively justified in practice.
    We lay the foundation for an alternative paradigm, arguing to instead report a data-dependent margin $\widehat{\Delta}_\alpha$ that bounds the true effect $\mu$ with probability $1 - \alpha$.
    Our key argument is that $\widehat{\Delta}_\alpha$ is more useful than a test outcome at a fixed margin $\Delta$, as measured by the guarantees it offers to decision makers.
    We generalize this to a curve of margins $\alpha \mapsto \widehat{\Delta}_\alpha$, uniformly valid under the post-hoc selection of the margin.
    % As a side-contribution, we show this spectrum may be easily aggregated across datasets and specifications.
    These ideas rely on e-values, which we derive for models that are strictly totally positive of order 3, nesting the classical z-test and t-test settings.
\end{abstract}

% \begin{abstract}
%     Equivalence testing concerns testing the null hypothesis that an effect is meaningfully large against the alternative that it is (near) zero.
%     Here, `meaningfully large' is expressed as the effect being larger than some  margin.
%     A longstanding problem is that this margin must be chosen in advance but is hard to objectively justify in most applications.
%     We provide a solution to this problem by studying equivalence testing under the post-hoc selection of the equivalence margin.

%     E-values are a continuous measures of evidence that generalize classical binary Neyman-Pearson tests, which may only reject-or-not.
%     Our solution is to report the e-value as evidence against each possible value of the margin.
%     We show that this may be interpreted as reporting a `fuzzy' generalization of a confidence set for the true effect.
%     These confidence sets may be easily aggregated across datasets across datasets and specifications, and may be sequentially updated with new data.
%     We base these fuzzy confidence sets on expected-utility-optimal e-values, which we derive for models that are totally positive of order 3, nesting the classical z-test and t-test settings.
% \end{abstract}

\newpage

\section{Introduction}
    % What is equivalence testing
    % The Equivalence Problem concerns establishing whether an effect size $\mu \geq 0$ (or $|\mu|$) is sufficiently small to be deemed `practically negligible'
    Equivalence assessment concerns the problem of establishing whether an effect size $\mu \geq 0$ (or $|\mu|$) is sufficiently small to be deemed practically negligible.
    This problem arises across many scientific disciplines, including clinical trials and bioequivalence assessment \citep{schuirmann1987comparison, berger1996bioequivalence},
        psychology \citep{lakens2018equivalence},
        political science \citep{hartman2018equivalence},
        economics \citep{dette2024testing},
        and ecological sciences \citep{robinson2004model}.

    % How does classical equivalence testing work?
    Equivalence assessment is almost universally cast into a hypothesis testing problem.
    Here, the analyst tests the null hypothesis $H_0^\Delta : \mu \geq \Delta$ that the effect $\mu$ is larger than some specified margin $\Delta$, against the alternative $H_1^\Delta : \mu < \Delta$ that it is smaller.
    By specifying $\Delta$ to be sufficiently small so that effect sizes $\mu < \Delta$ can be viewed as (practically) equivalent to zero, this results in the level $\alpha > 0$ Type-I error guarantee
    \begin{align}\label{ineq:type-I}
        \Pr(\textnormal{falsely claim equivalence})
            \equiv \Pr(\textnormal{falsely reject } \mu \geq \Delta)
            \leq \alpha.
    \end{align}

    % Problem with \Delta
    A longstanding problem with this approach is that the equivalence margin $\Delta$ is often difficult to justify in practice.
    At the same time, the formal Type-I error guarantee \eqref{ineq:type-I} is inherently tied to the choice of $\Delta$ through the null hypothesis $H_0 : \mu \geq \Delta$: it only certifies a claim of equivalence \emph{relative to the specified margin $\Delta$}.
    This creates an awkward tension: to obtain a formal guarantee one must commit to a margin $\Delta$ that may be hard to defend.
    
    % Is this the right approach?
    This tension invites a natural question: if the margin $\Delta$ is hard to meaningfully specify, is the Type-I error guarantee \eqref{ineq:type-I} relative to such a margin $\Delta$ actually what we want?
    That is, should equivalence assessment even be cast into a testing problem?
    
    \subsection{Reporting data-dependent equivalence margins}
        % We believe that in many equivalence assessment applications, the underlying goal is simply to provide evidence against $\mu$ being `large'.
        % By casting this into a hypothesis testing
        % By casting this problem into a hypothesis testing 
        % Casting this problem into a hypothesis testing problem makes it impossible to achieve this goal.

        The main point of this paper is that equivalence assessment should \emph{not} be cast into a testing problem.
        As a starting point, we instead suggest reporting a \emph{data-dependent equivalence margin} $\widehat{\Delta}_\alpha$ that satisfies the guarantee
        \begin{align}\label{ineq:guarantee}
            \Pr(\mu < \widehat{\Delta}_\alpha) \geq 1 - \alpha,
        \end{align}
        for a data-independent $\alpha > 0$.
        In words, $\widehat{\Delta}_\alpha$ may be viewed as a statistically certified upper bound on $\mu$, or as an upper confidence limit of a level $1 - \alpha$ confidence set $[0, \widehat{\Delta}_\alpha)$ for $\mu$.

        % While we are not the first to observe the possibility of reporting $\widehat{\Delta}_\alpha$ (Westlake, Seaman \& Serlin, Meyners), this has always been presented as an exploratory tool, `inferior' to a test outcome for a specified $\Delta$.
        % We explicitly argue that $\widehat{\Delta}_\alpha$ should replace the reporting of such a test outcome.

        To motivate reporting $\widehat{\Delta}_\alpha$, we argue that it is \emph{more useful} than a test outcome, as measured by the guarantees it offers to decision makers.
        In particular, suppose that the analyst hands $\widehat{\Delta}_\alpha$ to a decision maker facing a decision $d \in \mathcal{D}$ under a loss function $L_\mu$ that is non-decreasing in the unknown parameter $\mu$.
        Then, we show that the decision maker can use $\widehat{\Delta}_\alpha$ to translate the guarantee \eqref{ineq:guarantee} into the loss bound
        \begin{align}\label{ineq:P_loss_bound}
            \Pr\left(L_\mu(\widehat{d}) \leq L_{\widehat{\Delta}_\alpha}(\widehat{d})\right) \geq 1 - \alpha,
        \end{align}
        for \emph{every data-dependent decision} $\widehat{d}$.
        % The tightest bound is attained by the decision that minimizes the worst-case loss $\widehat{d}_\alpha^{\textnormal{MM}} \in \argmin_{d \in \mathcal{D}} L_{\widehat{\Delta}_\alpha}(d)$.

        Within this framework, we find that reporting a fixed-$\Delta$ test outcome is appropriate if the decision maker faces a loss $L_\mu$ that depends only on whether $\mu < \Delta$ or $\mu \geq \Delta$.
        However, the fact that $\Delta$ is hard to specify in practice suggests that practical decisions rarely hinge on a single fixed $\Delta$.
        As a consequence, we argue that a fixed-$\Delta$ test is generally not the appropriate methodology for equivalence assessment.
        
        We formalize this discussion in Section \ref{sec:data-dependent_margins} and Section \ref{sec:evidence-certified_decisions}.
        There, we show that a data-dependent margin $\widehat{\Delta}_\alpha$ is characterized by a specific \emph{collection} of tests.
        Moreover, we formally derive the loss bound \eqref{ineq:P_loss_bound}.
        Loss bounds such as \eqref{ineq:P_loss_bound} were recently studied by \citet{andrews2025certified} for classical confidence sets and by \citet{kiyani2025decision} for conformal prediction.
        
    \subsection{Choosing the equivalence margin post-hoc}
        We distinguish the use of a data-dependent margin $\widehat{\Delta}_\alpha$ from the \emph{post-hoc} selection of the margin.
        To explain this, note that it may be tempting to compute an entire \emph{equivalence curve} $\alpha \mapsto \widehat{\Delta}_\alpha$ of data-dependent equivalence margins and browse the curve to select the desired certificate, post-hoc.
        Unfortunately, this would break the guarantee \eqref{ineq:guarantee}, since this makes $\alpha$ data-dependent and \eqref{ineq:guarantee} only holds for \emph{data-independent} choices of $\alpha$.

        To resolve this, we consider recent innovations in the data-dependent and post-hoc selection of $\alpha$ \citep{grunwald2024beyond, koning2023post}.
        Using the approach of \citet{koning2023post}, we show how to obtain a \emph{uniformly valid} equivalence curve $\alpha \mapsto \widehat{\Delta}_\alpha$ that satisfies
        \begin{align}\label{ineq:post-hoc_margin}
            \E_{\widetilde{\alpha}}\left[\frac{\Pr(\mu \geq \widehat{\Delta}_{\widetilde{\alpha}} \mid \widetilde{\alpha})}{\widetilde{\alpha}}\right] \leq 1,
        \end{align}
        for \emph{every data-dependent level $\widetilde{\alpha}$}.
        This guarantee may be interpreted as rewriting \eqref{ineq:guarantee} as $\Pr(\mu \geq \widehat{\Delta}_\alpha) / \alpha \leq 1$ and upgrading from data-independent $\alpha$ to `in expectation over every data-dependent $\widetilde{\alpha}$' (\emph{not} conditionally on $\widetilde{\alpha} = \alpha$).

        As \eqref{ineq:post-hoc_margin} holds for arbitrarily data-dependent levels $\widetilde{\alpha}$, we may truly browse the entire equivalence curve $\alpha \mapsto \widehat{\Delta}_\alpha$ and select the desired equivalence margin $\widehat{\Delta}_{\widetilde{\alpha}}$, post-hoc.
        We display such a post-hoc valid equivalence curve in the left panel of Figure \ref{fig:marginplot}. 
        There, we also show a data-dependent margin $\widehat{\Delta}_\alpha$ at the fixed level $\alpha = 0.05$, which may be viewed as a special uniformly valid equivalence curve for which $\widehat{\Delta}_{\widetilde{\alpha}} = \infty$ for $\widetilde{\alpha} < 0.05$ and $\widehat{\Delta}_{\widetilde{\alpha}} = \widehat{\Delta}_\alpha$ for $\widetilde{\alpha} \geq \alpha$.
        We see that the fixed-$\alpha$ margin is tailored to $\alpha = 0.05$ while the smooth equivalence curve provides a balanced range of margins.
        
        We formalize this discussion in Section \ref{sec:e-values_post-hoc}.
        Moreover, in Section \ref{sec:evidence-certified_decisions} we use \eqref{ineq:post-hoc_margin} to generalize the loss bound \eqref{ineq:P_loss_bound} to post-hoc loss bounds, providing much more flexibility to decision makers, inspired by recent work of \citet{koning2025fuzzy} and \citet{grunwald2023posterior}.
        
        \begin{figure}[t]
            \centering
            \includegraphics[width = 0.47\linewidth]{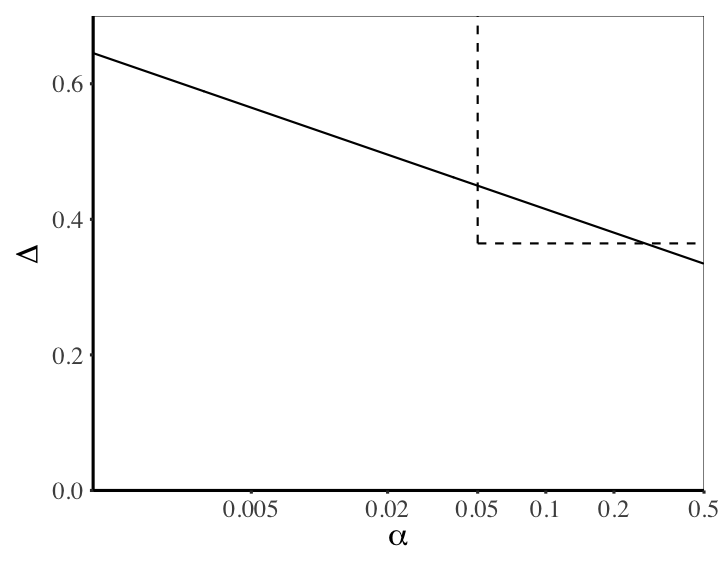}
            \includegraphics[width = 0.47\linewidth]{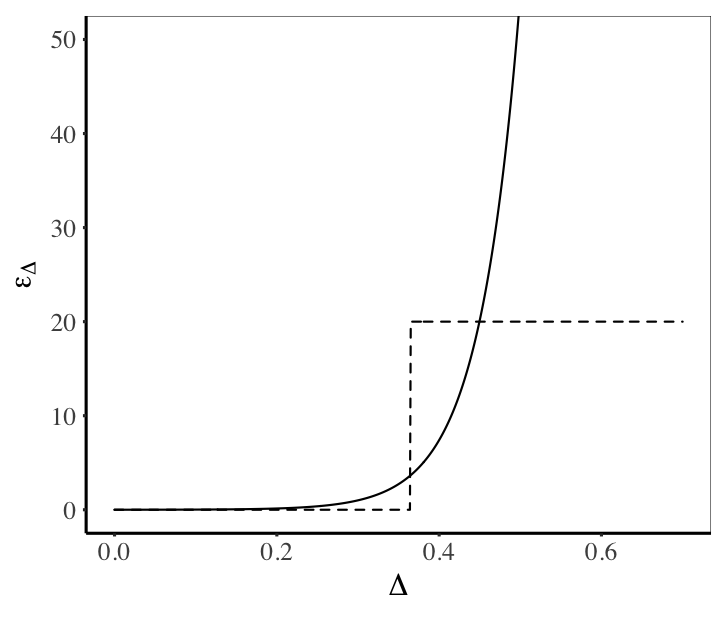}
            \caption{The left panel shows two uniformly valid equivalence curves $\alpha \mapsto \widehat{\Delta}_\alpha$, where the dashed line corresponds to a fixed-$\alpha$ data-dependent margin. The right panel re-expresses these as curves of e-values $\Delta \mapsto \e_\Delta$.} 
            \label{fig:marginplot}
        \end{figure}

    \subsection{Connection to e-values and merging}
        Under the hood, the data-dependent margin $\widehat{\Delta}_\alpha$ is derived by inverting a particular curve of classical equivalence tests $\Delta \mapsto \phi_\Delta$.
        To obtain uniformly valid equivalence curves $\alpha \mapsto \widehat{\Delta}_\alpha$, we generalize test inversion to inverting a curve of \emph{e-values} $\Delta \mapsto \e_\Delta$ for equivalence.
        The e-value is a recently introduced measure of evidence, and may be viewed as a continuous generalization of a classical hypothesis test \citep{howard2021time, shafer2021testing, vovk2021values, ramdas2023game, grunwald2024safe, koning2024continuous, ramdas2024hypothesis}.
        We illustrate this e-value inversion in Figure \ref{fig:marginplot}, where the right panel expresses the curves in the left panel as curves of e-values $\Delta \mapsto \e_\Delta$.
        There, the dashed line corresponds to a curve of tests $\Delta \mapsto \phi_\Delta$, which are $\{0, 1/\alpha\}$-valued e-values.
        The curve $\Delta \mapsto \phi_\Delta$ switches from $0$ to $1 / \alpha$ at $\widehat{\Delta}_\alpha$.

        A remarkable consequence is that we can easily merge equivalence curves $\alpha \mapsto \widehat{\Delta}_\alpha$ across studies (under independence) and across specifications (under arbitrary dependence) by translating them into their curve of e-values $\Delta \mapsto \e_\Delta$.
        Indeed, the product of two \emph{independent} curves of e-values $\Delta \mapsto \e_\Delta^1 \times \e_\Delta^2$ is also a valid curve of e-values. 
        Moreover, the weighted average of two \emph{arbitrarily dependent} curves of e-values $\Delta \mapsto w \e_\Delta^1 + (1-w) \e_\Delta^2$, $w \in [0, 1]$, remains a valid curve of e-values.
        We formally treat these merging properties in Section \ref{sec:merging}.
        
    \subsection{Technical contributions}
        Using the bridge between equivalence curves and e-values, we derive good equivalence curves by deriving optimal e-values for equivalence assessment.
        This constitutes our main technical contribution.

        For the optimality, we restrict ourselves to models that are strictly totally positive of order 3 (STP$_3$), which nests exponential families, the $z$-test, $t$-test and more.
        Under this assumption, we show that the log-utility-optimal e-value \citep{grunwald2024safe, larsson2025numeraire} has a closed-form boundary-mixture likelihood ratio.
        We recover a similar structure when generalizing to general utility optimal e-values, which nests classically power-optimal tests \citep{lehmann1959testing, romano2005} for the ``Neyman--Pearson utility function'' $U_\alpha(x) = x \wedge 1/\alpha$ \citep{koning2024continuous}.

        We also introduce the \emph{TOST-E}: the e-value generalization of the classical Two One-Sided Tests procedure \citep{schuirmann1981hypothesis}.
        We show it is valid under the weaker monotone likelihood ratio condition (STP$_2$).
        We also consider a (generalized) Universal Inference variant \citep{wasserman2020universal}, which is valid under no assumptions.
        We apply this theory to the $z$-test and $t$-test in Section \ref{sec:examples_z_t}.

        The optimality is formally derived in Section \ref{sec:optimalSTP3}, and the TOST-E and Universal Inference variants in Section \ref{sec:TOST-E}.
        Finally, in Section \ref{sec:seqtesting}, we develop anytime-valid sequential extensions of these constructions.
        
    \subsection{Relationship to literature}\label{sec:related}
        We provide a short timeline of how equivalence assessment was originally cast into a testing problem.
        Moreover, we relate our work to several papers that caution against reporting a data-dependent margin $\widehat{\Delta}_\alpha$, and present our counterarguments.
        
        To the best of our knowledge, testing equivalence with a fixed $\Delta$ was first considered in Chapter 3.7 of \citet{lehmann1959testing}.
        The general problem of assessing equivalence was studied by \citet{westlake1976symmetrical, westlake1979statistical} who effectively proposed reporting $\widehat{\Delta}_\alpha$.
        In an attempt to propose a `formal' methodology, \citet{dunnett1977significance} and \citet{blackwelder1982proving} cast this into the fixed-$\Delta$ testing framework of \citet{lehmann1959testing}.
        This fixed-$\Delta$ framework was subsequently consolidated by regulatory bodies for clinical trials, outlining that $\Delta$ should be prespecified in a transparent and conservative manner, based on ``[...] both statistical reasoning and clinical judgement [...]'' \citep{ICH_E9_1998, ICH_E10_2000}.
        A large body of literature has subsequently focused on how one should motivate $\Delta$ \citep{wiens2002choosing, lange2005choice, fitzgerald_2025}.
        Based on recent draft regulations of the \citet{EMA2025DraftNonInferiority}, the fixed-$\Delta$ testing approach remains the standard to this day.

        Following the original proposal by \citet{westlake1976symmetrical}, several other papers have discussed reporting a data-dependent margin $\widehat{\Delta}_\alpha$, based on inverting various tests.
        In particular, \citet{hauck1986proposal} propose overcoming the prespecification of $\Delta$, by plotting a (not uniformly valid) equivalence curve $\alpha \mapsto \widehat{\Delta}_\alpha$, as a ``useful way to present [...] the degree of certainty regarding potential differences''.
        Moreover, both \citet{seaman1998equivalence} and \citet{meyners2007least} propose reporting a variant of $\widehat{\Delta}_\alpha$ based on inverting the TOST of \citet{schuirmann1981hypothesis}.
        We stress that these works explicitly only view $\widehat{\Delta}_\alpha$ as an informal, exploratory, descriptive or diagnostic tool, and not as a replacement of the fixed-$\Delta$ testing methodology.

        Key arguments against viewing $\widehat{\Delta}_\alpha$ as a serious alternative to the fixed-$\Delta$ approach are presented in \citet{ng2003issues} and \citet{campbell2021make}.
        Both arguments rely on casting equivalence assessment back into a binary testing problem, showing that no meaningful Type-I error control can be achieved for the data-dependent hypothesis $H_{\widehat{\Delta}_\alpha}$.
        Indeed, \citet{campbell2021make} observe that $H_{\widehat{\Delta}_\alpha}$ is rejected by construction, leading to the conclusion that reporting $\widehat{\Delta}_\alpha$ ``[...] lacks the formalism of equivalence testing [...]''.
        \citet{ng2003issues} similarly argues that the data-dependent hypothesis $H_{\widehat{\Delta}_\alpha}$ would be rejected in an identical repeated experiment with probability 1/2, concluding that reporting $\widehat{\Delta}_\alpha$ is ``[...] exploratory and [...] unacceptable for confirmatory testing [...]''.

        Our main counterargument bypasses this discussion: we believe it is not relevant whether we can obtain a Type-I error control on rejecting $H_{\widehat{\Delta}_\alpha}$.
        Instead, we believe that one should simply compare statistical procedures by how useful they are in subsequent decision making.
        There, we argue that the guarantee \eqref{ineq:guarantee} on $\widehat{\Delta}_\alpha$ is more useful and no less `formal' than a Type-I error guarantee for a fixed $\Delta$.
        
        % Our main counterargument bypasses this discussion entirely: we believe that statistical procedures should not be judged by a perception of `formalism', but instead by how useful they are in subsequent decision making.
        % Here, we show that reporting the data-dependent margin $\widehat{\Delta}_\alpha$ or its post-hoc valid generalization comes out on top.
        % Moreover, even if we were to dive into a discussion of formalism: we see no reason to view the guarantee \eqref{ineq:guarantee} as any less formal than \eqref{ineq:type-I}.
        
%     \subsection{Outline}
    
% The remainder of the paper is organized as follows.
% Section~2 introduces the e-value formulation of equivalence assessment and develops test/e-value inversion for obtaining data-dependent, post-hoc-valid margin curves.
% Building on this inferential layer, Section~3 translates margin guarantees into loss guarantees and certified (minimax) decisions, including post-hoc decision rules.
% Section~4 then derives optimal e-values under STP$_3$, covering both log-optimal (numeraire / GRO) and general utility-optimal targets.
% Section~5 presents the TOST-E benchmark under STP$_2$ and its comparison with the optimal construction.
% Section~6 illustrates the framework in the canonical Gaussian settings (symmetric $z$-test and $t$-test/effect-size setting).
% Section~7 extends the theory to sequential, anytime-valid procedures.
% Section~8 gives a practical decision-theoretic example, and Section~9 concludes with a discussion.

    \section{Background: from tests to e-values}\label{sec:ebackground}
        We denote a sample space by $\cX$, which we equip with a model $\mathcal{P}$: a collection of probabilities on $\cX$.
        A hypothesis $H \subseteq \mathcal{P}$ is a subset of our model.
        The intention of hypothesis testing is to obtain evidence against a hypothesis $H$.
        
        Without loss of generality, we follow \citet{koning2024continuous} in defining a hypothesis test $\phi^\alpha$ as a $\{0, 1/\alpha\}$-valued map.
        Here, $\phi^\alpha = 0$ is interpreted as a non-rejection and $\phi^\alpha = 1/\alpha$ as a rejection at level $\alpha > 0$.
        This may be interpreted as emitting either no evidence (0) or a particular amount of evidence ($1 / \alpha$) against a hypothesis.
        \begin{dfn}[Test]
            A level-$\alpha$ test is a measurable map $\phi^\alpha : \cX \to \{0, 1/\alpha\}$.
            A test $\phi^\alpha$ is valid for hypothesis $H$ if $ \E^{P}[\phi^\alpha] \equiv P(\phi^\alpha = 1/\alpha) / \alpha \leq 1, \textnormal{ for every } P \in H$.
        \end{dfn}

        The \emph{e-value} can be viewed as a `multi-significance level' generalization of a test.
        Indeed, the e-value extends the binary codomain $\{0, 1/\alpha\}$ of a test to the richer $\{0, 1/\alpha_1, 1/\alpha_2, \dots\}$, $\alpha_1, \alpha_2, \dots > 0$, or even $[0, \infty]$.
        This means an e-value may return various amounts of evidence against the hypothesis.
        In particular, the realization of an e-value $\e$ may be interpreted as a rejection at level $1/\e$ under a generalized Type-I error guarantee \citep{koning2023post}. 
        \begin{dfn}[E-value]
            An e-value is a measurable map $\e : \cX \to [0, \infty]$.
            An e-value $\e$ is valid for hypothesis $H$ if $\E^{P}[\e] \leq 1, \textnormal{ for every } P \in H$.
        \end{dfn}

    \section{Equivalence assessment}\label{sec:data-dependent_margins}
        In this section, we derive the equivalence between data-dependent margins $\widehat{\Delta}_\alpha$ and particular curves of tests $\Delta \mapsto \phi_\Delta$, and generalize this to uniformly valid equivalence curves $\alpha \mapsto \widehat{\Delta}_\alpha$ and particular curves of e-values $\Delta \mapsto \e_\Delta$.

        For now, we focus on a single margin $\widehat{\Delta}_\alpha$ instead of the pair of margins $(\widehat{\Delta}_\alpha^-, \widehat{\Delta}_\alpha^+)$ typical in equivalence testing, because such a pair introduces considerable notational overhead without adding conceptual insight.
        The single-margin setting is also known in the literature as assessing `non-inferiority' \citep{wellek2010equivalence}.
        We cover margin pairs in Section \ref{sec:twomargins}.

        \subsection{Model and hypotheses}
            To formally set up assessing equivalence, we consider a model $(P_\mu)_{\mu \geq 0}$ indexed by a parameter $\mu$.
            We use $\mu^*$ to denote the parameter associated to the `true' data generating process $P_{\mu^*}$.
            We consider a collection of hypotheses $(H_0^\Delta)_{\Delta \geq 0}$ of the form $H_0^\Delta = \{P_{\mu} : \mu \geq \Delta\}$.
            Here, the hypothesis $H_0^\Delta$ may be interpreted as the statement that $\mu^*$ is at least $\Delta$.

            By the structure of the problem, the index $\Delta^*$ of the smallest true hypothesis $H_0^{\Delta^*}$ coincides with $\mu^*$, since $\Delta^* = \inf\{\Delta \in [0, \infty) : \mu^* \leq \Delta\} = \mu^*$.
            To avoid causing confusion by frequently switching back and forth between $\Delta$ and $\mu$, we use this identity to formulate the problem as obtaining a bound on $\Delta^*$.

        \subsection{Data-dependent margins}
            In theory, we would desire a probabilistic data-dependent bound $\widehat{\Delta}_\alpha$ on $\Delta^*$ that is valid under the smallest true hypothesis: $\sup_{\mu \geq \Delta^*} P_\mu(\Delta^* \geq \widehat{\Delta}_\alpha) \leq \alpha$.
            However, as the true value $\Delta^*$ is unknown, and we do not wish to take a stance on its value, we achieve this by requiring this inequality to hold for \emph{every} margin $\Delta$.
            This leads us to Definition \ref{dfn:valid_dd_margin}.

            \begin{dfn}[Valid data-dependent margin]\label{dfn:valid_dd_margin}
                A margin $\widehat{\Delta}_\alpha$ is valid at level $\alpha > 0$ if
                \begin{align}\label{ineq:guarantee_formal}
                    \sup_{\Delta \geq 0} P_\Delta(\Delta \geq \widehat{\Delta}_\alpha) \leq \alpha.
                \end{align}
            \end{dfn}

            \begin{rmk}
                The guarantee \eqref{ineq:guarantee_formal} is equivalent to $\sup_{\Delta \geq 0} \sup_{\mu \geq \Delta} P_\mu(\Delta \geq \widehat{\Delta}_\alpha)
                        \leq \alpha$.%, which may be interpreted as $\sup_{H_0^\Delta \ni \mu} \sup_{P \in H_0^\Delta} P(H_0^\Delta \subseteq H_0^{\widehat{\Delta}_\alpha}) \leq \alpha$, for every $\mu \geq 0$.
            \end{rmk}

            \begin{rmk}\label{rmk:confidence_set}
                The data-dependent margin $\widehat{\Delta}_\alpha$ may be interpreted as the upper bound of a confidence set $[0, \widehat{\Delta}_\alpha)$ for $\Delta^*$. 
            \end{rmk}

        \subsection{Data-dependent margins and test inversion}\label{sec:test-inversion}
            In Proposition \ref{prp:Delta_test_equivalence}, we show that $\widehat{\Delta}_\alpha$ is uniquely characterized by a non-decreasing curve $\Delta \mapsto \phi_\Delta^\alpha$ of equivalence tests for $H_0^\Delta$.
            This means that $\widehat{\Delta}_\alpha$ and $\Delta \mapsto \phi_\Delta^\alpha$ can be viewed as two different representations of the same information.
            Moreover, it implies that to construct $\widehat{\Delta}_\alpha$, we can reduce to the familiar task of constructing an equivalence test $\phi_\Delta^\alpha$, for every hypothesis $H_0^\Delta$.
            The proof of Proposition \ref{prp:Delta_test_equivalence} is found in Appendix \ref{proof:Delta_test_equivalence}, which also carefully handles the measurability of $\widehat{\Delta}_\alpha$.
            
            The precise correspondence between the data-dependent margin and curve of tests is that the margin $\widehat{\Delta}_\alpha$ induces the tests $\phi_\Delta^\alpha = \mathbb{I}\{\Delta \geq \widehat{\Delta}_\alpha\} / \alpha$.
            Conversely, the data-dependent margin can be retrieved by inverting a curve of tests $\Delta \mapsto \phi_\Delta^\alpha$ through $\widehat{\Delta}_\alpha = \inf\{\Delta : \phi_\Delta^\alpha = 1/\alpha\}$.
            
            \begin{prp}\label{prp:Delta_test_equivalence}
                A data-dependent margin $\widehat{\Delta}_\alpha$ corresponds to a non-decreasing right-continuous curve $\Delta \mapsto \phi_\Delta^\alpha$ of tests.
                $\widehat{\Delta}_\alpha$ is valid if and only if $\phi^{\alpha}_\Delta$ is valid for $H_0^\Delta$ for every $\Delta \geq 0$.
            \end{prp}

            \begin{rmk}[Non-decreasing in $\Delta$]\label{rmk:non-decreasing}
                Inverting a collection of tests that is not non-decreasing in $\Delta$ may result in a confidence set that is not an interval of the form $[0, \widehat{\Delta}_\alpha)$.
                While there is nothing fundamentally wrong with such a confidence set for $\Delta^*$, it is harder to interpret, because it assigns more evidence against smaller values of $\Delta$ (larger hypotheses) than against larger values of $\Delta$ (smaller hypotheses).
                
                In Section \ref{sec:optimalSTP3}, we find that optimal tests are generally non-decreasing in $\Delta$.
                Moreover, we can convert any family of tests into a non-decreasing family by taking its right-lower envelope $\underline{\phi}_\Delta^\alpha(x) := \inf_{\Delta' \geq \Delta} \phi^\alpha_{\Delta'}(x)$.
            \end{rmk}

        \subsection{Uniform validity, post-hoc margins and e-value inversion}\label{sec:e-values_post-hoc}
            We now generalize beyond tests to e-values to construct uniformly valid equivalence curves $\alpha \mapsto \widehat{\Delta}_\alpha$.
            In particular, we consider a non-decreasing collection $\Delta \mapsto \e_\Delta$ of e-values $\e_\Delta$ valid for $H_0^\Delta$.
            We then generalize test inversion to e-value inversion by constructing a non-increasing equivalence curve through
            \begin{align*}
                \widehat{\Delta}_\alpha
                    = \inf\{\Delta : \e_\Delta \geq 1/\alpha\}.
            \end{align*}

            The key contribution here is that this does not merely result in a curve of \emph{individually valid} margins $\widehat{\Delta}_\alpha$.
            Instead, the curve $\alpha \mapsto \widehat{\Delta}_\alpha$ is \emph{uniformly valid}, as defined in Definition \ref{dfn:post-hoc_valid}.
            This result is presented in Theorem \ref{thm:post-hoc_valid}, which is proven in Appendix \ref{proof:post-hoc_valid}.

            The consequence of uniform validity is that we may interpret the entire equivalence curve $\alpha \mapsto \widehat{\Delta}_\alpha$ as a whole.
            For example, it enables the \emph{post-hoc selection} of the margin: we may browse the curve $\alpha \mapsto \widehat{\Delta}_\alpha$ and select the data-dependent level $\widetilde{\alpha}$ that comes with the desired value of $\widehat{\Delta}_{\widetilde{\alpha}}$.
            
            \begin{dfn}\label{dfn:post-hoc_valid}
                We say that the equivalence curve $\alpha \mapsto \widehat{\Delta}_\alpha$ is uniformly valid if
                \begin{align*}
                    \sup_{\Delta \geq 0} \E_{\widetilde{\alpha}}^{P_\Delta}\left[\frac{P_\Delta(\Delta \geq \widehat{\Delta}_{\widetilde{\alpha}} \mid \widetilde{\alpha})}{\widetilde{\alpha}}\right]
                        \leq 1,
                \end{align*}
                 for every data-dependent choice of the level $\widetilde{\alpha} > 0$.
            \end{dfn}

            \begin{thm}\label{thm:post-hoc_valid}
                A non-increasing and right-continuous equivalence curve $\alpha \mapsto \widehat{\Delta}_\alpha$ corresponds to a non-decreasing and right-continuous collection of e-values $\Delta \mapsto \e_\Delta$.
                The equivalence curve is uniformly valid if and only if every e-value $\e_\Delta$ is valid for $H_0^\Delta$.
            \end{thm}

            \begin{rmk}
                $\Delta \mapsto \e_\Delta$ can be obtained from the equivalence curve: $\e_\Delta = \sup\left\{\frac{1}{\alpha} : \Delta \geq \widehat{\Delta}_\alpha\right\}$.
            \end{rmk}

            \begin{rmk}
                As e-values generalize tests, the discussion here is a (rich) generalization of Section \ref{sec:test-inversion}.
                In particular, fixing the level $\alpha = a$, a data-dependent margin $\widehat{\Delta}_a$ corresponds to the equivalence curve $\alpha \mapsto \widetilde{\Delta}_\alpha$ for which $\widetilde{\Delta}_\alpha = \infty$ for $\alpha < a$ and $\widetilde{\Delta}_\alpha = \widehat{\Delta}_a$ for $\alpha \geq a$.
            \end{rmk}

            \begin{rmk}[Evidence non-decreasing in $\Delta$]\label{rmk:non-decreasing_e-value}
                Remark \ref{rmk:non-decreasing} extends to e-values.
                Indeed, if $\Delta \mapsto \e_\Delta$ were not non-decreasing, then we could have more evidence against larger hypotheses than against the smaller hypotheses nested within them.
                If necessary, we can also take the right-lower envelope to enforce this: $\underline{\e}_\Delta(x) := \inf_{\Delta' \geq \Delta} \e_{\Delta'}(x)$.
            \end{rmk}

        \subsection{Merging equivalence curves}\label{sec:merging}
            Merging evidence is important in equivalence assessment: evidence for the same problem may come from multiple independent studies, or multiple analyses based on different specifications may be performed on the same data.
            One of the attractive features of e-values is that they enable easy merging of evidence \citep{vovk2021values}.
            Using the link between curves of e-values and equivalence curves, we extend these merging operations to equivalence curves.

            In particular, consider two non-decreasing curves of valid e-values, $\Delta \mapsto \e_\Delta^j$, $j = 1, 2$.
            Their weighted average $\Delta \mapsto w \e_\Delta^1 + (1 - w) \e_\Delta^2$ is a curve of valid e-values, $w \in [0, 1]$.
            If both are independent, then their product $\Delta \mapsto \e_\Delta^1 \e_\Delta^2$ is a curve of valid e-values.
            By Theorem \ref{thm:post-hoc_valid}, inverting this merged curve of e-values produces a merged uniformly valid equivalence curve.
            In Proposition \ref{prp:mergingcurves}, we show this merging can be expressed in terms of equivalence curves.

            \begin{prp}\label{prp:mergingcurves} 
                Let $\alpha \mapsto \widehat{\Delta}_\alpha^1$ and $\alpha \mapsto \widehat{\Delta}_\alpha^2$ be two uniformly valid equivalence curves.
                Then, the following equivalence curve is uniformly valid:
                \begin{align*}
                    \widehat{\Delta}_\alpha^w
                        =   \inf_{\substack{
                                \alpha_1, \alpha_2 > 0: \\
                                \frac{w}{\alpha_1} + \frac{1 - w}{\alpha_2} = \frac{1}{\alpha}}
                            }
                            \max\{\widehat{\Delta}_{\alpha_1}^1,\widehat{\Delta}_{\alpha_2}^2\}.
                \end{align*}
                If the curves are independent, then the following curve is also uniformly valid:
                \begin{align*}
                    \widehat{\Delta}_\alpha^\times
                    = \inf_{\substack{\alpha_1, \alpha_2 > 0 : \\ \alpha_1 \alpha_2 = \alpha}}
                        \max\{\widehat{\Delta}_{\alpha_1}^1, \widehat{\Delta}_{\alpha_2}^2\}.
                \end{align*}
            \end{prp}

    \subsection{Two margins}\label{sec:twomargins}
        Up to this point, we have focused on a single margin to facilitate the presentation of our ideas, expressed in terms of one-sided hypotheses $H_0^\Delta : \mu \geq \Delta$.
        We now turn to the two-sided version of equivalence testing, expressed by hypotheses of the form
        \begin{align*}
            H_0^{(\Delta^-, \Delta^+)} : \mu \leq \Delta^- \textnormal{ or } \mu \geq \Delta^+,
        \end{align*}
        for a pair of margins $\Delta^- \leq \Delta^+$, and a real-valued parameter $\mu \in \mathbb{R}$.
        This recovers the one-sided setting if we set $\Delta^- = -\infty$, or $\Delta^- = 0$ if $\mu \in \mathbb{R}_+$.

        For assessing equivalence in this two-sided setting, we may consider an e-value $\varepsilon_{\Delta^-, \Delta^+}$ for each hypothesis $H_0^{(\Delta^-, \Delta^+)}$, and so for each pair of margins $(\Delta^-, \Delta^+)$.
        A resulting curve of e-values $(\Delta^-, \Delta^+) \mapsto \varepsilon_{\Delta^-, \Delta^+}$ is then assumed to be non-decreasing coordinate-wise: in the partial order $\precsim$ that is defined by $(\Delta_1^-, \Delta_1^+) \precsim (\Delta_2^-, \Delta_2^+)$ if and only if both $\Delta_1^- \geq \Delta_2^-$ and $\Delta_1^+ \leq \Delta_2^+$.
        That is, $\varepsilon_{\Delta_1^-, \Delta_1^+} \leq \varepsilon_{\Delta_2^-, \Delta_2^+}$ if $(\Delta_1^-, \Delta_1^+) \precsim (\Delta_2^-, \Delta_2^+)$.

        To obtain the two-margin analogue of an equivalence curve $\alpha \mapsto \widehat{\Delta}_\alpha$, we can similarly invert a curve of e-values.
        Here, we do not merely obtain a pair $(\widehat{\Delta}_\alpha^-, \widehat{\Delta}_\alpha^+)$ of margins for each $\alpha$.
        Instead, we obtain a collection of pairs for every $\alpha$:
        \begin{align*}
            \widehat{\mathcal{D}}_\alpha
                = \inf\{(\Delta^-, \Delta^+) : \varepsilon_{\Delta^-, \Delta^+} \geq 1 / \alpha\},
        \end{align*}
        where the infimum here is interpreted to produce the set of maximal lower bounds in the partial order $\precsim$.
        
    \section{Evidence-certified decisions}\label{sec:evidence-certified_decisions}
        In this section, we motivate our call to report a data-dependent margin $\widehat{\Delta}_\alpha$, or a more general uniformly valid curve of margins $\alpha \mapsto \widehat{\Delta}_\alpha$, instead of a classical fixed-$\Delta$ equivalence test outcome $\phi_\Delta$.
        For this purpose, we study the guarantees provided by these objects to decision makers.

        In particular, we consider a decision maker who faces a decision $d \in \mathcal{D}$ based on a loss function $L_{\mu} : \mathcal{D} \to [0, \infty)$ that is non-decreasing and left-continuous in the (unknown) true value $\mu^*$ of $\mu$, for every decision $d \in \mathcal{D}$.
        The decision maker receives information about $\mu^*$ in the form of a data-dependent margin $\widehat{\Delta}_\alpha$, or equivalence curve $\alpha \mapsto \widehat{\Delta}_\alpha$, to aid in the decision making process.

        \begin{exm}
            The decision maker may be a regulator who needs to make a decision $d$ on the recommended daily consumption of a certain food item, based on the unknown concentration $\mu^*$ of a potentially harmful substance in the food item.
            Here, the loss $L_\mu$ may describe the expected negative impact on public health associated with each decision if $\mu$ is the true concentration, where higher concentrations result in a larger negative impact. 
            A test $\phi_\Delta$, data-dependent margin $\widehat{\Delta}_\alpha$ or equivalence curve $\alpha \mapsto \widehat{\Delta}_\alpha$ expresses the information about $\mu^*$ that is available to the decision maker.
        \end{exm}

        \subsection{Decisions based on a data-dependent margin}
            In Theorem \ref{thm:evidence-informed_decision}, we show how the validity \eqref{ineq:guarantee_formal} of a data-dependent margin $\widehat{\Delta}_\alpha$ can be passed on to a uniform loss bound across data-dependent decisions.
            This bound is expressed by the loss function $L_{\widehat{\Delta}_\alpha}$ at $\mu = \widehat{\Delta}_\alpha$.
            The interpretation of this result is that whatever the true value $\mu^*$ is, the probability that a decision $\widehat{d}$ yields a loss greater than $L_{\widehat{\Delta}_\alpha}(\widehat{d})$ is small.
            
            \begin{thm}[Uniform loss bound]\label{thm:evidence-informed_decision}
                If $\widehat{\Delta}_\alpha$ is valid, then $\sup_{\mu \geq 0} P_{\mu}(L_{\mu}(\widehat{d}) > L_{\widehat{\Delta}_\alpha}(\widehat{d})) \leq \alpha$, for every data-dependent decision $\widehat{d}$.
            \end{thm}
            \begin{proof}
                As $\mu \mapsto L_\mu(d)$ is non-decreasing for each $d$, we have $L_\mu(\widehat{d}) > L_{\widehat{\Delta}_\alpha}(\widehat{d}) \implies \widehat{\Delta}_\alpha \leq \mu$.
                Validity of $\widehat{\Delta}_\alpha$ then implies $\sup_{\mu \geq 0} P_{\mu}(L_{\mu}(\widehat{d}) > L_{\widehat{\Delta}_\alpha}(\widehat{d})) \leq \sup_{\mu \geq 0} P_{\mu}(\mu \geq \widehat{\Delta}_\alpha) \leq \alpha$.
            \end{proof}

            This loss bound is inspired by recent work on certifying decisions by \citet{andrews2025certified}.
            They focus on the decision that attains the tightest bound: the minimax decision $\widehat{d}_\alpha^{\textnormal{MM}} \in \argmin_{d \in \mathcal{D}} \sup_{\mu < \widehat{\Delta}_\alpha} L_{\mu}(d) = \argmin_{d \in \mathcal{D}} L_{\widehat{\Delta}_\alpha}(d)$, assuming it exists.
            Such a minimax decision is also known as an \emph{as-if} decision, because the decision is made `as-if' $\mu^* < \widehat{\Delta}_\alpha$  \citep{manski2021econometrics}.

            Our Theorem \ref{thm:evidence-informed_decision} is more than an application of \citet{andrews2025certified} to equivalence assessment.
            The innovation here is that our result holds uniformly across all data-dependent decisions, whereas \citet{andrews2025certified} only consider the minimax decision.
            This is important, because even though the minimax decision attains the tightest bound, decision makers may opt for a different decision for external reasons.
            Our result still provides a loss bound on such decisions.
            Moreover, it avoids the technical problem of requiring the existence of a minimax decision.

            % \begin{rmk}
            %     Instead of the plain minimax decision, Wang \& Dobriban consider weighting $L_{\widehat{\Delta}_\alpha}$ with the worst-case loss if $\mu > \Delta$: $L_\infty(d) = \lim_{\mu \to \infty} L_\mu(d)$.
            %     This leads to a minimax decision $\argmin_{d \in \mathcal{D}} w L_{\widehat{\Delta}_\alpha}(d) + (1-w) L_\infty(d)$, $L_\infty(d) = \lim_{\mu \to \infty} L_\mu(d)$.
            % \end{rmk}

        \subsection{When is a classical equivalence test sufficient?}\label{sec:recover_classical_testing} 
            The loss bound provided in Theorem \ref{thm:evidence-informed_decision} provides a strong motivation for reporting a data-dependent margin $\widehat{\Delta}_\alpha$.
            At the same time, this analysis raises a natural question: for what kind of loss function would the outcome of a classical fixed-$\Delta$ equivalence test be sufficient?
            In this section, we show that a classical equivalence test $\phi_\Delta$ is sufficient if the loss function hinges on a fixed $\Delta$.

            Consider a loss function $L_\mu$ for which the loss $L_\mu(d)$ of every decision only depends on whether $\mu \leq \Delta$.
            That is, for every $d \in \mathcal{D}$ there exist values $\ell^-(d) \leq \ell^+(d)$ such that 
            \begin{align*}
                L_\mu(d) = \ell^-(d) \mathbb{I}\{\mu \leq \Delta\} + \ell^+(d) \mathbb{I}\{\mu > \Delta\}.
            \end{align*}
            This means that the bound $L_{\widehat{\Delta}_\alpha}(\widehat{d})$ on the loss $L_{\mu}(\widehat{d})$ in Theorem \ref{thm:evidence-informed_decision} only depends on the data through the outcome of the test $\phi_\Delta = \mathbb{I}\{\widehat{\Delta}_\alpha \leq \Delta\} / \alpha$.
            Indeed, substituting this loss into Theorem \ref{thm:evidence-informed_decision} yields the loss bound
            \begin{align*}
                \sup_{\mu \geq 0} P_\mu\left(L_{\mu}(\widehat{d}) > \ell^-(\widehat{d}) \mathbb{I}\{\widehat{\Delta}_\alpha \leq \Delta\} + \ell^+(\widehat{d}) \mathbb{I}\{\widehat{\Delta}_\alpha > \Delta\}\right)
                    \leq \alpha.
            \end{align*}
            As a result, the outcome of the test $\phi_\Delta$ is the relevant information for the decision maker.

            To see this more directly, we can show that the minimax decision only hinges on the outcome of $\phi_\Delta$.
            Indeed, for this loss function we have that the optimal decision $d^*(\mu) \in \argmin_{d \in \mathcal{D}} L_\mu(d)$ only depends on whether $\mu \leq \Delta$: $d^*(\mu) = \argmin_{d \in \mathcal{D}} \ell^-(d) =: d^-$ if $\mu \leq \Delta$ and $d^*(\mu) = \argmin_{d \in \mathcal{D}} \ell^+(d) =: d^+$ if $\mu > \Delta$.
            Substituting $\widehat{\Delta}_\alpha$ in for $\mu$ yields that the minimax decision equals $\widehat{d}_\alpha^{\textnormal{MM}} = d^-$ if $\phi_\Delta$ rejects and $\widehat{d}_\alpha^{\textnormal{MM}} = d^+$, otherwise.

            This analysis shows the outcome of a single fixed-$\Delta$ equivalence test $\phi_\Delta$ is relevant if the loss $L_\mu$ hinges on this value of $\Delta$.
            At the same time, the fact that $\Delta$ is hard to specify in practice suggests that loss functions rarely hinge on a single value of $\Delta$ in practice.
            Indeed, if practical loss functions would often hinge on a single $\Delta$, then we believe practitioners would not struggle as much to pick $\Delta$.
            This supports our conclusion that a single test outcome $\phi_\Delta$ is not the right statistical object to report, and that reporting $\widehat{\Delta}_\alpha$ is more appropriate.
        
        \subsection{Decisions based on equivalence curves}
            In Theorem \ref{thm:uniform_curve_decisions}, we generalize Theorem \ref{thm:evidence-informed_decision} to uniformly valid equivalence curves $\alpha \mapsto \widehat{\Delta}_\alpha$.
            In this setting, the loss bound is now represented by a spectrum of loss functions $\alpha \mapsto L_{\widehat{\Delta}_\alpha}$ that is uniformly valid in the confidence certificate $\alpha$.

            \begin{thm}\label{thm:uniform_curve_decisions}
                If $\alpha \mapsto \widehat{\Delta}_\alpha$ is uniformly valid, then
                \begin{align*}
                    \sup_{\mu \geq 0} \E_{\widetilde{\alpha}}^{P_\mu}\left[\frac{P_\mu(L_\mu(\widehat{d}) > L_{\widehat{\Delta}_{\widetilde{\alpha}}}(\widehat{d}) \mid \widetilde{\alpha})}{\widetilde{\alpha}}\right]
                        \leq 1,
                \end{align*}
                for every data-dependent decision $\widehat{d}$ and data-dependent level $\widetilde{\alpha}$.
            \end{thm}
            
            There are various ways to use this loss spectrum:
            \begin{itemize}
                \item Given an arbitrary data-dependent decision $\widehat{d}$, it provides a spectrum $\alpha \mapsto L_{\widehat{\Delta}_\alpha}(\widehat{d})$ of loss bounds on $L_\mu(\widehat{d})$, each coupled with a different confidence level $\alpha$.
                \item Given a post-hoc choice of the margin $\Delta$, we can retrieve the corresponding data-dependent confidence level $\widetilde{\alpha} = \inf\{\alpha : \Delta \geq \widehat{\Delta}_\alpha\}$ from the equivalence curve $\alpha \mapsto \widehat{\Delta}_\alpha$ to obtain the loss function $d \mapsto L_{\widehat{\Delta}_{\widetilde{\alpha}}}(d)$.
                \item Given a post-hoc choice of the confidence level $\widetilde{\alpha}$, we can retrieve the corresponding margin $\widehat{\Delta}_{\widetilde{\alpha}}$ and obtain the matching loss function $d \mapsto L_{\widehat{\Delta}_{\widetilde{\alpha}}}(d)$.
            \end{itemize}

            One concrete idea is to present the decision maker with an entire spectrum of minimax decisions and their associated loss bounds: $\alpha \mapsto (\widehat{d}_\alpha^{\textnormal{MM}}, L_{\widehat{\Delta}_\alpha}(\widehat{d}_\alpha^{\textnormal{MM}}))$, where $\widehat{d}_\alpha^{\textnormal{MM}} \in \argmin_{d \in \mathcal{D}} L_{\widehat{\Delta}_\alpha}(d)$.
            A decision maker may then browse such decisions and select the decision with the desired loss bound \citep{koning2025fuzzy}.

            \begin{rmk}
                An alternative idea proposed by \citet{grunwald2023posterior} is to de-emphasize implausible values of $\Delta$ by weighting the $L_\Delta$ by the amount of evidence $\e_\Delta$ against $\Delta$.
                This leads to weighted minimax decisions $\argmin_{d \in \mathcal{D}} \sup_{\Delta \geq 0} L_\Delta(d) / \e_\Delta$, for non-negative loss functions $L_\Delta$.
            \end{rmk}

    \section{Optimal e-values for equivalence under STP$_3$}\label{sec:optimalSTP3}
        In Section \ref{sec:data-dependent_margins}, we established uniform validity of equivalence curves obtained by inverting e-values.
        However, validity alone is not sufficient to guarantee that a curve is informative.
        Indeed, a curve for which $\widehat{\Delta}_\alpha = \infty$ for every $\alpha$ is also uniformly valid, but completely uninformative.
        Because these curves are obtained by inverting a curve of e-values $\Delta \mapsto \e_\Delta$, the remainder of this paper focuses on studying good e-values.
        
        In this section, we switch to the two-sided equivalence setting, covering hypotheses of the form $H_0^{(\Delta^-,\Delta^+)}: \mu \leq \Delta^- \ \text{or}\ \mu \geq \Delta^+$ as discussed in Section \ref{sec:twomargins}.
        We make this switch because the optimality results for the two-sided setting are conceptually richer and technically more demanding than the one-sided claims.\footnote{The one-sided results can be viewed as a limiting special case, obtained by taking $\Delta^- \to -\infty$ which leads to setting $c = 0$ in Theorem \ref{thm:Uoptimal}, only requiring STP$_2$.}

    \subsection{Setup and STP$_3$}\label{sec:modelsetup}
        We consider a single-parameter model $(P_\mu)_{\mu \in \cM}$ dominated by a $\sigma$-finite measure $\nu$, where $\cM \subseteq \mathbb{R}$.
        We consider a sample space $\mathcal{X} \subseteq \mathbb{R}$, which may be interpreted as the sample space of a (sufficient) statistic.
        We denote the density of $P_\mu$ with respect to $\nu$ by $p_\mu = dP_\mu / d\nu$.

        Fix $\Delta^- < \Delta^+$ with $\Delta^-, \Delta^+ \in \mathcal{M}$.
        We consider the hypotheses,
        \begin{align*}
            H_0^{(\Delta^-, \Delta^+)}
                = \{P_\mu : \mu \leq \Delta^- \text{ or } \mu \geq \Delta^+\},
            \quad
            H_1^{(\Delta^-, \Delta^+)} 
                = \{P_\mu : \Delta^- < \mu < \Delta^+\},
        \end{align*}
        assuming that both are non-empty.

        We make one structural assumption: we assume strict total positivity of order 3 (STP$_3$) of the kernel $(\mu,x)\mapsto p_\mu(x)$ \citep{karlin1968total}.
        The formal definition of STP$_3$ is quite technical, and so deferred to Appendix~\ref{app:tp3counterexample}.
        STP$_3$ is satisfied, for example, by exponential families.

        \begin{rmk}[STP$_3$, STP$_2$ and MLR]
            STP$_3$ may be interpreted as a strengthened version of the Monotone Likelihood Ratio assumption (MLR), which coincides with STP$_2$.
            Where MLR controls pairwise likelihood-ratio monotonicity, STP$_3$ additionally controls a triplewise ordering. 
            This additional structure is precisely what delivers the variation-diminishing properties used below \citep{brown1981variation}.
        \end{rmk}

    \subsection{Neyman-Pearson optimality}
        We start by briefly covering the familiar Neyman-Pearson optimality framework.

        Under STP$_3$ a uniformly most powerful (UMP) test is known to exist.
        Its critical region is determined by a likelihood ratio against a mixture of the two boundary densities \citep{lehmann1959testing, kallenberg1984}. 
        Concretely, for any $P_\mu \in H_1^{(\Delta^-, \Delta^+)}$, the classical level-\(\alpha\) UMP test can be written as a special e-value
        \begin{equation*}
            \e^\alpha(x)
            =
            \frac{1}{\alpha}\,\mathbb{I}\!\{\Lambda_c(x)>k_\alpha\}
            +
            \frac{\eta_\alpha(x)}{\alpha}\,\mathbb{I}\!\{\Lambda_c(x)=k_\alpha\},
            \qquad
            \Lambda_c(x):=
            \frac{p_{\mu}(x)}
            {c\,p_{\Delta^-}(x)+(1-c)\,p_{\Delta^+}(x)},
        \end{equation*}
        where $c\in[0,1]$, $k_\alpha>0$, and $\eta_\alpha:\cX\to[0, 1]$ are chosen to satisfy the boundary size equalities  $\E^{P_{\Delta^-}}[\e^\alpha(X)]=\E^{P_{\Delta^+}}[\e^\alpha(X)] = 1$.

        Here, the case that $\Lambda_c(x) = k_\alpha$ is classically interpreted as an instruction to reject with probability $\eta_\alpha(x)$ using external randomization.
        However, we follow \citet{koning2024continuous} in recommending to instead report $\e^\alpha(x)$ directly as evidence.

    %With these assumptions in place, we now directly characterize utility-optimal e-values. 
    
    \subsection{Utility-optimal e-values}\label{sec:Uoptimal}
        The classical Neyman-Pearson framework can be viewed as a special case of a more general optimal e-value setting.
        In particular, a natural goal is to select an e-value that maximizes the expected utility $\E^{P_\mu}[U(\e)]$ under some alternative $P_\mu \in H_1^{(\Delta^-, \Delta^+)}$ for some utility function $U : [0, \infty] \to [-\infty, \infty]$.
        Indeed, this can be shown to recover the Neyman-Pearson framework for the `Neyman-Pearson utility function' $U_\alpha(e) = \min\{e, 1 / \alpha\}$ \citep{koning2024continuous}.

        Beyond the Neyman-Pearson utility function, one can generally not expect an e-value to be expected-utility optimal over a composite hypothesis $H_1^{(\Delta^-, \Delta^+)}$.
        We therefore instead maximize the expected utility $\E^Q[U(\e)]$ under a mixture alternative $Q = \int P_\mu\, dw(\mu)$, where $w$ is a mixture over $H_1^{(\Delta^-, \Delta^+)}$.
        In practice, $w$ may be chosen as a Dirac measure on some relevant point (e.g. $w = \delta_0$ if $\mu = 0$ is a relevant alternative), or as some kind of uniform distribution over the interval $[\Delta^-, \Delta^+]$.

        Let $\cE$ denote the class of valid e-values for $H_0^{(\Delta^-, \Delta^+)}$.
        Assume that $U : [0, \infty] \to [-\infty, \infty]$ is increasing, strictly concave and continuously differentiable on $(0, \infty)$.
        Moreover, assume that its derivative $U'$ satisfies $\lim_{x \to \infty} U'(x) = 0$.
        To ensure that the optimizer exists without restrictions on the distributions, we assume $x \mapsto x U'(x)$ is bounded.

        Theorem~\ref{thm:Uoptimal} presents our main optimality result.
        As in the classical Neyman-Pearson case, these utility-optimal e-values also involve a boundary-mixture likelihood ratio.

        \begin{thm}\label{thm:Uoptimal}
            Assume \(\{p_\mu:\mu\in\cM\}\) is STP\(_3\), and let \(q\) be the density of the mixture alternative \(Q\) with respect to $\nu$. For \(c\in(0,1)\) and \(\lambda>0\), define
            \[
            \e_{c,\lambda}(x):=(U')^{-1}\!\left(
            \lambda\,
            \frac{c\,p_{\Delta^-}(x)+(1-c)\,p_{\Delta^+}(x)}{q(x)}
            \right).
            \]
            Then there exists \((c^*,\lambda^*)\in(0,1)\times(0,\infty)\) such that \(\e_{c^*,\lambda^*}\) is the \(U\)-optimal e-variable, i.e.
            \[
            \e_{c^*,\lambda^*}\in\arg\max_{\e\in\cE}\E_Q[U(\e)].
            \]
        \end{thm}

        Corollary \ref{cor:logoptimal} covers the popular log-utility \citep{grunwald2024safe, larsson2025numeraire}.
        \begin{cor}[Log-optimal e-value]\label{cor:logoptimal}
            Under \(U(e)=\log e\), Theorem~\ref{thm:Uoptimal} yields
            \[
            \e_{c^*}(x)=\frac{q(x)}{c^*p_{\Delta^-}(x)+(1-c^*)p_{\Delta^+}(x)}.
            \]
        \end{cor}

    \subsection{Comments on the proof of Theorem~\ref{thm:Uoptimal}}
        The proof of Theorem~\ref{thm:Uoptimal} rests on two STP$_3$-based ingredients. The first is a calibration result: it determines $(c,\lambda)$ by ensuring that $\e_{c,\lambda}$ has expectation exactly $1$ at both boundary points. The second is a shape result: it ensures that $x \mapsto \e_{c,\lambda}(x)$ has at most one local extremum, which is what allows boundary calibration to extend to the entire composite null.
        
        \begin{prp}\label{prp:determineconstantutility}
            There exists $(c^*,\lambda^*)\in(0,1)\times(0,\infty)$ such that $\E_{\Delta^-}[\e_{c^*,\lambda^*}(X)] = 1$ and $\E_{\Delta^+}[\e_{c^*,\lambda^*}(X)] = 1$.
        \end{prp}
        
        \begin{prp}\label{prp:unimodality}
        For every $c\in(0,1)$ and $\lambda>0$, $x \mapsto \e_{c,\lambda}(x)$ has at most one local extremum.
        \end{prp}
        
        \begin{rmk}[Weakening to STP$_2$]
            The STP$_3$ assumption in Theorem~\ref{thm:Uoptimal} cannot, in general, be weakened to STP$_2$ (MLR). In particular, there are STP$_2$ families for which the likelihood ratio $\e_c$
            with $c$ chosen to satisfy the boundary size conditions is not a valid e-variable for the full composite null. A counterexample is given in Appendix~\ref{app:tp3counterexample}: STP$_2$ holds but STP$_3$ fails, and the resulting candidate violates e-value validity at another null parameter value.
        \end{rmk}

    \section{TOST and Universal Inference}\label{sec:TOST-E}

        \subsection{The TOST-E}
            Perhaps the most widely used test for equivalence testing is the Two One-Sided Tests (TOST) procedure \citep{schuirmann1981hypothesis, schuirmann1987comparison}.
            As the name suggests, the idea is to test both $ H_0^L:\mu\le \Delta^-$ and $H_0^R:\mu\ge \Delta^+$ at level $\alpha$, and reject the equivalence null $H_0^{(\Delta^-, \Delta^+)}$ if both are rejected.

            Generalizing from tests to e-values, we introduce the TOST-E procedure based on two one-sided e-values $\e^L$ and $\e^R$ for $H_0^L$ and $H_0^R$.
            Since e-values produce a continuous amount of evidence, the natural generalization is to take their minimum: $  \e^{\mathrm{TOST}}(x) := \min\{\e^L(x),\e^R(x)\}$.
            As $\e^L$ and $\e^R$ are valid for $H_0^L$ and $H_0^R$, this minimum is indeed valid for their union $H_0^{(\Delta^-, \Delta^+)}$.

            In the expected-utility framework, a natural choice for the one-sided e-values is
            $\e^L(x) := (U')^{-1}(p_{\Delta^-}(x)/q(x))$ and
            $\e^R(x) := (U')^{-1}(p_{\Delta^+}(x)/q(x))$.
            Under STP$_2$ (MLR), these are valid for $H_0^L$ and $H_0^R$, so that their combined TOST-E is valid under STP$_2$.

            \begin{rmk}[Comparison log-optimal and TOST-E]
                Taking $U = \log$, the TOST-E can be written as $\e^{\text{TOST}}(x) =  q(x) / \max\{p_{\Delta^-}(x), p_{\Delta^+}(x)\}$.
                Comparing this to the log-optimal e-variable $\e^{\text{log}}(x) = q(x) / (c^*p_{\Delta^-}+(1-c^*)p_{\Delta^+})$, we find $\e^{\text{log}}(x) \geq \e^{\text{TOST}}(x)$ for every $x \in \mathcal{X}$, since $\max\{p_{\Delta^-}(x), p_{\Delta^+}(x)\} \geq c^*p_{\Delta^-}+(1-c^*)p_{\Delta^+}$.
                As a result, the log-optimal e-value dominates the TOST-E, at the cost of requiring STP$_3$ over STP$_2$.
            \end{rmk}

    \subsection{Universal inference}
        If $(P_\mu)_{\mu \in \cM}$ is not STP$_2$, an option is to rely on (a generalization of) Universal Inference \citep{wasserman2020universal}, which is valid without assumptions on the model $(P_\mu)_{\mu \in \cM}$.

        Instead of taking the minimum of two one-sided e-values for the boundary hypotheses $H_0^L : \mu \leq \Delta^-$ and $H_0^R : \mu \geq \Delta^+$, the (generalized) Universal Inference approach takes the essential infimum\footnote{The greatest \emph{measurable} lower bound. See Appendix A.2 in \citet{ramdas2022admissible}.} $\text{essinf}_{\mu \in \cM_0} \e_\mu$ over a collection of e-values $(\e_\mu)_{\mu \in \cM_0}$ where $\e_\mu$ is valid for the hypothesis $\{P_\mu\}$ and $\cM_0 = \{\mu : P_\mu \in H_0^{\Delta^-, \Delta^+}\}$.
        While this approach may seem overly conservative, every valid e-value can be interpreted as a special case of generalized Universal Inference (see e.g. Proposition 3 in \citet{koning2025sequentializing}).

        \begin{rmk}[Comparison TOST-E and UI]
            The name Universal Inference is usually reserved for an infimum over log-optimal / likelihood ratio e-values $\e_\mu = q(x) / p_\mu(x)$.
            The resulting e-value $\e^{\text{UI}} = \textnormal{essinf}_{\mu \in \cM_0} q(x) / p_\mu(x) = q(x) / \textnormal{esssup}_{\mu \in \cM_0} p_\mu(x)$ is dominated by its TOST-E counterpart $q(x) / \max\{p_{\Delta^-}, p_{\Delta^+}\}$, since $\textnormal{esssup}_{\mu \in \cM_0} p_\mu(x) \geq \max\{p_{\Delta^-}, p_{\Delta^+}\}$.
        \end{rmk}

    \section{Examples: $z$-test \& $t$-test}\label{sec:examples_z_t}
    % We illustrate our optimal e-values in the Gaussian location setting.

% We now illustrate the preceding constructions in the two canonical Gaussian settings. We first consider the fixed-variance $z$-test, where the symmetric-margin numeraire is available in closed form. We then consider the unknown-variance $t$-test, where the situation differs: for inference on the raw mean $\mu$, a direct numeraire construction is generally unavailable, while for standardized effect sizes one can recover a numeraire-type likelihood-ratio approach.

        \subsection{$z$-test}
            Let $X_1, \dots, X_n \overset{iid}{\sim} \mathcal{N}(\mu,\sigma^2)$ with known variance $\sigma^2$.  
            We fix the margins $\Delta^- < \Delta^+$ and consider the hypotheses $H_0^{(\Delta^-,\Delta^+)}:\ \mu\le \Delta^-$ or $\mu \ge \Delta^+$ against a mixture alternative $Q$ on $\mu \in (\Delta^-,\Delta^+)$ with mixing distribution $w$.
            Let $v:=\sigma^2/n$, write $\bar X$ for the sample mean, and let $\varphi_v(\cdot)$ denote the density of $\mathcal N(0,v)$. Then the alternative density is given by $q(\bar x)=\int_{\Delta^-}^{\Delta^+}\varphi_v(\bar x-\mu)\,w(d\mu).$
            By Corollary~\ref{cor:logoptimal}, the log-optimal e-value is
            \[
            \e_{c^*}(\bar X)
            =
            \frac{q(\bar X)}
            {c^*\,\varphi_v(\bar X-\Delta^-)+(1-c^*)\,\varphi_v(\bar X-\Delta^+)},
            \]
            where $c^*\in(0,1)$ is uniquely determined by
            $\E_{\Delta^-}[\e_{c^*}(X)] = \E_{\Delta^+}[\e_{c^*}(X)]=1$.
            
            \begin{rmk}
                If $w$ is symmetric around $(\Delta^-+\Delta^+)/2$ then $c^*=1/2$.
            \end{rmk}

        \subsection{$t$-test}\label{sec:ttest}
            Consider $X_1,\ldots,X_n \stackrel{iid}{\sim}\mathcal N(\mu,\sigma^2)$ with unknown $\sigma^2$.
            Here, we should distinguish two objectives: assessing equivalence of the mean $\mu$ or of the standardized effect size $\delta := \mu / \sigma$. 
            Empirically, both targets are common: mean-scale margins are standard in bioequivalence and other applications with natural measurement units \citep{westlake1976symmetrical, schuirmann1987comparison}, whereas standardized-effect margins are widely used when unit-free comparability across studies or outcomes is the primary goal \citep{lakens2017equivalence}.

%\paragraph{Equivalence on the mean scale.}
%With unknown $\sigma$, this is a two-parameter problem, so the one-parameter STP$_3$ route from Section~\ref{sec:optimalSTP3} does not apply directly. A practical choice is therefore applying TOST-E on $\mu$.

\paragraph{Equivalence on the mean scale.}
With unknown $\sigma$, a valid e-value $\e$ must satisfy
$\sup_{|\mu|\ge\Delta,\ \sigma>0}\mathbb{E}_{\mu,\sigma}[\e]\le 1$.
Even after reduction to sufficient statistics, $\sigma$ remains a nuisance, so enforcing this uniformly is typically conservative. TOST-E addresses this by splitting $H_0$ into two one-sided components.
For the right component, with $\delta^+:=(\mu-\Delta^+)/\sigma$, the set
$\{(\mu,\sigma):\mu\ge\Delta^+,\ \sigma>0\}$ becomes the one-sided hypothesis $\delta^+\ge 0$, testable by one-sided $t$-based e-values (and similarly on the left).

For this purpose, define one-sided $t$-statistics
$T_L:=\sqrt{n}(\bar X-\Delta^-)/S$ and
$T_R:=\sqrt{n}(\Delta^+-\bar X)/S$,
with $\nu=n-1$.
Let $f_{0,\nu}$ be the central $t_\nu$ density, and let $q_L,q_R$ be one-sided alternative densities for $T_L,T_R$ (e.g. noncentral-$t$ mixtures on positive noncentralities). Then
$
\e_\mu^L(X):=q_L(T_L)/f_{0,\nu}(T_L),$ and $
\e_\mu^R(X):=q_R(T_R)/f_{0,\nu}(T_R),
$
which yields the TOST-E:
\[
\e_\mu^{\mathrm{TOST}}(X):=\min\{\e_\mu^L(X),\e_\mu^R(X)\}.
\]

\paragraph{Equivalence on the standardized effect size scale.}
Here the nuisance scale is absorbed by standardization, and equivalence assessment reduces to a one-parameter problem in $\delta$.
Letting $f_\delta$ denote the density of $T_n:=\sqrt{n}\bar X/S$ under effect size $\delta$ (noncentral $t$ with noncentrality $\sqrt{n}\,\delta$), and $w$ be a mixing distribution on $(\Delta^-,\Delta^+)$, define
$
q(t):=\int_{\Delta^-}^{\Delta^+} f_u(t)\,w(du).
$
The corresponding boundary-mixture form is
\[
\e_{c^*}(t):=
\frac{q(t)}{c^*\,f_{\Delta^-}(t)+(1-c^*)\,f_{\Delta^+}(t)},
\]
with $c^*$ calibrated by the two boundary equalities.
%For symmetric margins and symmetric $w$, this yields $c^*=1/2$.

\paragraph{Plotting equivalence curves.}
An illustration of several methods to assess equivalence on the mean $\mu$ in the Gaussian setting is given in Figure~\ref{fig:zt_tost_compare} for a single data realization.
The left panel reports the equivalence curves $\alpha \mapsto \widehat{\Delta}_\alpha$ and the right panel the corresponding curves of e-values $\Delta \mapsto \e_\Delta$.

\begin{figure}[H]
    \centering
    \includegraphics[width=0.95\linewidth]{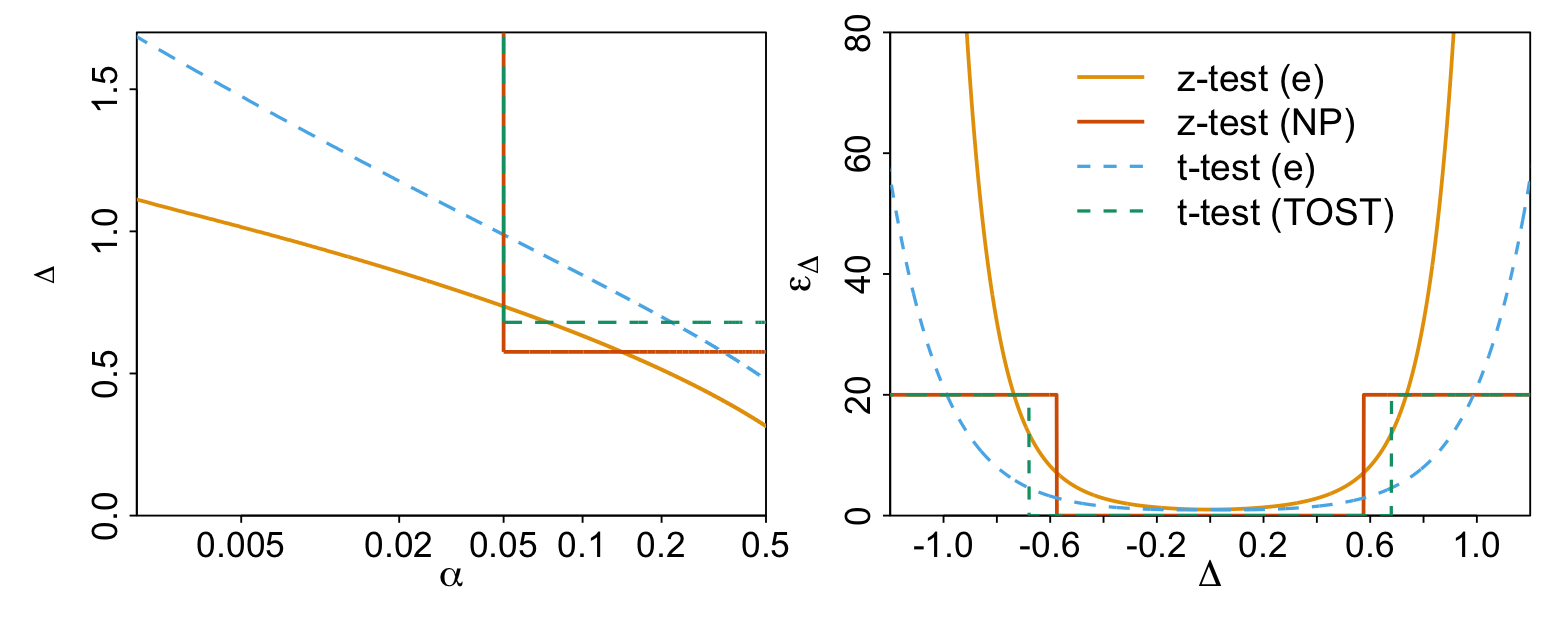}
    \vspace{-0.15cm}
    \caption{Comparison of four procedures to assess equivalence on the mean $\mu$ in the Gaussian location model. Left: positive half of equivalence curves $\alpha \mapsto \widehat{\Delta}_\alpha$. Right: curves of e-values $\Delta \mapsto \e_\Delta$. 
    The setup for the displayed realization is: $\bar X=0.05$, $n=40$ and $\sigma=1$.}
    \label{fig:zt_tost_compare}
\end{figure}
    \vspace{-0.8cm}

    \section{Sequential equivalence assessment}\label{sec:seqtesting}

        \subsection{(Adaptive) anytime validity}
            In many applications of equivalence assessment, data arrives sequentially: data may arrive over time within a study, or we may observe a sequence of studies on the same hypothesis.
            For this purpose, we consider the sequential generalization of the e-value: the \emph{e-process}.
    
            \begin{dfn}[E-process]
                Consider the sample space $(\mathcal{X}, \mathcal{F})$ equipped with the filtration $(\mathcal F_t)_{t \geq 0}$, $\mathcal{F}_t \subseteq \mathcal{F}$, $t \geq 0$.
                We say that $(\e^t)_{t \geq 0}$ is an e-process if it is adapted to $(\mathcal F_t)_{t \geq 0}$ and $\e^t$ is $[0, \infty]$-valued.
            \end{dfn}

            E-processes come in two main flavors: anytime valid and adaptively anytime valid.
            Plain anytime validity may be interpreted as assuming that the data-generating process is fixed at the start.
            Adaptive anytime validity is a more stringent condition, and allows the data-generating process to adaptively change over time.
            
            \begin{dfn}[Anytime validity]
                An e-process $(\e^t)_{t \geq 0}$ is anytime valid for hypothesis $H$ if $\e^\tau$ is a valid e-value for $H$, for every stopping time $\tau$ adapted to $\mathbb{F}$.
            \end{dfn}
    
            \begin{dfn}[Adaptive anytime validity]
                An e-process is adaptively anytime valid for hypothesis $H$ if $\E^P[\e^\tau \mid \mathcal{F}_\sigma] \leq \e^\sigma$, for every $P \in H$ and every pair of stopping times $\sigma \leq \tau$ adapted to $\mathbb{F}$, and $\e^0 = 1$.
            \end{dfn}
    
            \begin{rmk}
                The term `e-process' is often reserved for what we call an anytime valid e-process.
                An adaptively anytime valid e-process is often called a `test supermartingale' \citep{shafer2011test, ramdas2022admissible, ramdas2023game}. 
            \end{rmk}
    
            Throughout this section, we focus on adaptively anytime valid e-processes.
            The main reason is that these are easier to construct, by relying on Proposition \ref{prp:test_martingale}.
            Indeed, it suffices to construct a conditional e-value $\e_s$ at every point in time $s \geq 1$, and subsequently construct the e-process by tracking their running product $\e^t = \prod_{s = 0}^t \e_s$, $\e_0 = 1$. The proof can be found in Appendix \ref{app:proofprptestmart}.
            
            \begin{dfn}[Conditional e-value]
                Let $\Sigma \subseteq \mathcal{F}$ be a sub-$\sigma$-algebra of $\mathcal{F}$.
                We say an e-value $\e$ is $\Sigma$-conditionally valid for hypothesis $H$ if $\E^P[\e \mid \Sigma] \leq 1$ a.s., for every $P \in H$.
            \end{dfn}
            
            \begin{prp}\label{prp:test_martingale}
                Let $\e_0 = 1$.
                $(\e^t)$ is adaptively anytime valid if and only if $\e^t = \prod_{s = 0}^t \e_s$, with $\mathcal{F}_{s - 1}$ conditionally valid and $\mathcal{F}_s$-measurable e-values $\e_s$, $s \geq 1$.
            \end{prp}

        \subsection{One-sided hypothesis}\label{sec:one-sided_sequential}
            We start with the one-sided setting $H_0^\Delta : \mu \geq \Delta$.
            Let $T_s$ denote a real-valued sufficient statistic for $\mu$ at time $s$.
            Here, the (log-optimal) likelihood ratio e-process becomes
            \[
            \e^t_{\Delta}
            =
            \prod_{s=1}^t
            \frac{q\!\left(T_s \mid \mathcal F_{s-1}\right)}
                 {p_{\Delta}\!\left(T_s \mid \mathcal F_{s-1}\right)},
            \]
            Collecting such e-processes over $\Delta$ yields a curve of e-processes $\Delta \mapsto (\e_\Delta^t)_{t \geq 0}$.
            
            Under a suitable STP$_2$ (MLR) condition in $T_s$, this e-process is adaptively anytime valid \citep{grunwald2025supermartingales}.
            \citet{koning2025sequentializing} extend the argument %of \citet{grunwald2025supermartingales} 
            beyond log-optimal e-values, showing that it suffices that $\e^t$ is monotone in the sufficient statistic.

\subsection{Two-sided symmetric hypotheses}

We now consider equivalence margins symmetric around zero, with margin $\Delta_S>0$. Specifically, we test
$H_0^S:\mu\le -\Delta_S$ or $\mu\ge \Delta_S$
against
$H_1^S:-\Delta_S<\mu<\Delta_S$.
This can be recast as a hypothesis on $\mu^2$:
\begin{equation*}\label{eq:squaredhypothesis_mu}
H_0^S:\ \mu^2\ge \Delta_S^2
\quad\text{vs.}\quad
H_1^S:\ \mu^2<\Delta_S^2.
\end{equation*}
Thus, the symmetric problem falls within the one-sided framework of the previous subsection, with boundary $\mu^2=\Delta_S^2$. We therefore construct boundary-calibrated one-sided conditional e-values for $\mu^2$ and track their running product $\e^t$, yielding an adaptively anytime-valid e-process. This can also be interpreted as an invariance-based reduction \citep{perez2024statistics}, since $\mu^2$ is invariant under $\mu\mapsto-\mu$. We apply this first to the symmetric $z$-test and then to the symmetric $t$-test.

\paragraph{Symmetric $z$-test}
Suppose $X_1,\ldots,X_n \stackrel{iid}{\sim}\mathcal N(\mu,1)$. By \citet{hall1965relationship}, the statistic $T_n:=\bar X_n^2$ is (invariantly) sufficient for $\mu^2$, and
$
T_n \sim \frac{1}{n}\chi^2_{1,n\mu^2}.
$
Let $p_{n,a}$ denote the density of $T_n$ under $\mu^2=a$, and let $q_n$ be any mixture of $(p_{n,a})_{a\in[0,\Delta_S^2]}$. Consider
\[
\e^n_{\Delta_S}:=\frac{q_n(T_n)}{p_{n,\Delta_S^2}(T_n)}.
\]
Because $(p_{n,a})_{a\ge 0}$ satisfies STP$_2$ (MLR) in $T_n$, this is an adaptively anytime valid e-process for the composite null $H_0^S$ by Theorem 4 of \cite{grunwald2025supermartingales}. Moreover, under Corollary 3 of \citet{perez2024statistics}, this e-statistic is log-optimal \citep{grunwald2024safe}.

\paragraph{Symmetric $t$-test}
We now return to Section~\ref{sec:ttest}, where $X_1,\ldots,X_n \stackrel{iid}{\sim}\mathcal N(\mu,\sigma^2)$ with unknown $\sigma^2$ and equivalence is assessed on the effect size $\delta=\mu/\sigma$. We consider the symmetric null
$H_0^S:\delta\le -\Delta_S$ or $\delta\ge \Delta_S$ (equivalently, $\delta^2\ge \Delta_S^2$).
Define
\[
T_n^2:=\frac{n\bar X_n^2}{S_n^2},
\qquad
S_n^2:=\frac{1}{n-1}\sum_{i=1}^n (X_i-\bar X_n)^2.
\]
To justify sufficiency for $\delta^2$, we follow \citet{hall1965relationship}. Let the group act by sign/scale transformations, $g_a(x_1,\ldots,x_n)=(ax_1,\ldots,ax_n)$ with $a\in\mathbb R\setminus\{0\}$, and define the maximal invariant by ratios $U_i=X_i/X_1$ (equivalently $U=(U_2,\ldots,U_n)$, with $U_1=1$). This invariant is unchanged under both sign changes and scaling. Since this group admits a Haar measure, Assumption A of \citet{hall1965relationship} is satisfied, so Stein’s theorem applies and implies that $T_n^2$ is (invariantly) sufficient for $\delta^2$.

Under $\delta^2=a$, \(T_n^2\sim F_{1,n-1,na}\); write $p^F_{n,a}$ for its density. Let $q_n$ be any mixture of $(p^F_{n,a})_{a\in[0,\Delta_S^2)}$, and define
\[
\e^n_{\Delta_S}:=\frac{q_n(T_n^2)}{p^F_{n,\Delta_S^2}(T_n^2)}.
\]
This is an adaptively anytime valid e-process for $H_0^S$, as the noncentral $F$ family is STP$_\infty$ in its noncentrality parameter \citep{brown1981variation}.
Appendix~\ref{app:comparetostesymt} provides simulation evidence showing that $\e^n_{\Delta_S}$ is more powerful than the sequential TOST-E.

\subsection{General two-sided testing}

We now return to the general (possibly asymmetric) equivalence problem with boundaries $\Delta^-<\Delta^+$, i.e. $H_0^{\Delta^-, \Delta^+} : \mu\le \Delta^-$ or $\mu\ge \Delta^+$. In contrast to the symmetric case, there is no direct reduction through sign invariance. We discuss three possible approaches here to construct adaptively anytime valid e-processes. 

\paragraph{Universal inference}
Let $\mathcal{M}_0 := \{\mu : \mu \leq \Delta^- \text{ or } \mu \geq \Delta^+\}$.
For each $\mu_0 \in \mathcal{M}_0$, let $(\e_{\mu_0}^t)_{t \geq 0}$ be an
anytime-valid e-process for the simple null hypothesis $P_{\mu_0}$, and define $\e_{\mathrm{UI}}^t := \text{ess\,inf}_{\mu_0 \in \mathcal{M}_0} \e_{\mu_0}^t$.
For every stopping time $\tau$ and every $\mu_0 \in \mathcal{M}_0$, $
    \mathbb{E}_{\mu_0}[\e_{\mathrm{UI}}^\tau]
    \leq \mathbb{E}_{\mu_0}[\e_{\mu_0}^\tau] \leq 1,$
so $(\e_{\mathrm{UI}}^t)_{t \geq 0}$ is anytime valid for $H_0^{\Delta^-, \Delta^+}$.

\paragraph{TOST-E}
Let $(\e_L^t)_{t\ge 0}$ and $(\e_R^t)_{t\ge 0}$ be anytime-valid e-processes for
$H_0^L:\mu\le \Delta^-$ and $H_0^R:\mu\ge \Delta^+$, respectively (see Section \ref{sec:one-sided_sequential}).
Define $\e_{\mathrm{TOST}}^t:=\min\{\e_L^t,\e_R^t\}$.
Fix any stopping time $\tau$ and $\mu_0\in\mathcal M_0$.
If $\mu_0\le \Delta^-$, then $\e_{\mathrm{TOST}}^\tau\le \e_L^\tau$, so
$\E_{\mu_0}[\e_{\mathrm{TOST}}^\tau]\le \E_{\mu_0}[\e_L^\tau]\le 1$.
Applying the same argument for $\mu \geq \Delta^+$ yields anytime validity of  $(\e_{\mathrm{TOST}}^t)_{t\ge 0}$ for $H_0^{\Delta^-, \Delta^+}$. 

\paragraph{Multiplying numeraires}
Under STP$_3$, one-step numeraires have a boundary-mixture form. This leads to a natural sequential construction:
\(
\e_s^{\mathrm{mix}}=q_s/\{c_s p_{\Delta^-,s}+(1-c_s)p_{\Delta^+,s}\}
\)
and
\(
\e_{\mathrm{mix}}^t=\prod_{s=1}^t \e_s^{\mathrm{mix}}.
\)
Under conditional validity, this yields adaptive anytime validity. Appendix~\ref{app:comparetostenum} indicates that, in our settings, sequential TOST-E is typically more powerful over time than the multiplied-numeraire process.

\section{Discussion}\label{sec:discussion} 
    We start by emphasizing that we do not believe there is anything inherently \emph{wrong} with classical fixed-$\Delta$ equivalence tests, nor with how they are interpreted.
    Instead, our main point is that they are simply not as \emph{useful} as the alternatives that we present, when viewed through the lens of decision making.
    Indeed, we believe statistical procedures should be compared based on the guarantees they can provide to decision making, and that one should be careful to not automatically cast every problem into a testing problem.

    Our key policy recommendation is that regulators should stop demanding a test outcome with respect to a fixed margin $\Delta$.
    Indeed, our arguments show that regulators should prefer receiving a data-dependent margin $\widehat{\Delta}_\alpha$, unless their loss hinges on such a fixed $\Delta$.
    A uniformly valid equivalence curve $\alpha \mapsto \widehat{\Delta}_\alpha$ offers even more flexibility in decision making.

    For two-sided equivalence assessment, we find that it is perhaps easier to report the curve of e-values $(\Delta^-, \Delta^+) \mapsto \e_{\Delta^-, \Delta^+}$.
    The proposal to report such a curve of e-values as evidence echoes recent calls in the e-value literature to report an e-value against each option of an unknown quantity \citep{grunwald2023posterior, koning2025fuzzy, koning2025sequentializing}.

    An open question is how to directly express the optimality of equivalence tests and e-values to their corresponding data-dependent margins and equivalence curves.

\subsection*{Data availability}\vspace{-0.2cm}
The figures may be replicated by running the R code available at the repository \url{https://github.com/StanKoobs/E-quivalenceTesting}.

{
\singlespacing
\small
\bibliographystyle{plainnat}
\bibliography{references}
}
\appendix

\section{Proofs for Section \ref{sec:data-dependent_margins}}

    We start with a technical measurability lemma.

    \begin{lemma}[Measurability of inversion]\label{lem:invert_threshold}
        Let $(\cX, \mathcal{A})$ and $(\mathcal{Y}, \mathcal{B})$ be measurable spaces.
        Let $g : \cX \times \mathbb{R}_+ \to [-\infty, \infty]$ satisfy:
        \begin{itemize}
            \item for every $\Delta \geq 0$, the map $x \mapsto g(x, \Delta)$ is $\mathcal{A}$-measurable,
            \item for every $x \in \cX$, the map $\Delta \mapsto g(x, \Delta)$ is non-decreasing.
        \end{itemize}
        Let $h : \mathcal{Y} \to [-\infty,\infty]$ be $\mathcal{B}$-measurable and define
        \begin{align*}
            \widehat{\Delta}(x,y)
                := \inf\{\Delta \geq 0 : g(x,\Delta) \geq h(y)\},
        \end{align*}
        where $(\inf \emptyset = \infty)$.
        Then $(x,y) \mapsto \widehat{\Delta}(x,y)$ is $\mathcal{A} \otimes \mathcal{B}$-measurable.
    \end{lemma}

    \begin{proof}
        Fix $t \geq 0$.
        Using the monotonicity in $\Delta$,
        \begin{align}\label{eq:countable_union}
            \{(x,y) : \widehat{\Delta}(x,y) < t\}
                = \bigcup_{q \in \mathbb{Q} \cap [0, t)} \{(x,y) : g(x,q)\ge h(y)\}.
        \end{align}
        Indeed, $\widehat{\Delta}(x,y) < t$ implies that there exists some $\Delta_0 < t$ with $g(x, \Delta_0) \geq h(y)$.
        Select $q \in \mathbb{Q} \cap [\Delta_0, t)$, which is possible since $\mathbb{Q}$ is dense.
        By monotonicity, $q \geq \Delta_0$ implies $g(x, q) \geq g(x, \Delta_0) \geq h(y)$.
        This means $\{(x,y) : \widehat{\Delta}(x,y) < t\}
                \subseteq \bigcup_{q \in \mathbb{Q} \cap [0, t)} \{(x,y) : g(x,q)\ge h(y)\}$.
        Conversely, if $g(x, q) \geq h(y)$ for some $q < t$, then $\widehat{\Delta}(x, y) \leq q < t$, so that $\{(x,y) : \widehat{\Delta}(x,y) < t\}
                \supseteq \bigcup_{q \in \mathbb{Q} \cap [0, t)} \{(x,y) : g(x,q)\ge h(y)\}$.

        Now, fix $q \in \mathbb{Q}$.
        Then, $(x, y) \mapsto g(x, q)$ is $\mathcal{A} \otimes \mathcal{B}$-measurable because it only depends on $x$.
        Likewise, $(x, y) \mapsto h(y)$ is also $\mathcal{A} \otimes \mathcal{B}$-measurable.
        Hence the set $\{(x, y): g(x, q) \geq h(y)\}$ is $\mathcal{A} \otimes \mathcal{B}$-measurable. 
        
        Now, as the union on the right-hand-side \eqref{eq:countable_union} is countable it is $\mathcal{A} \otimes \mathcal{B}$-measurable.
        Hence, the left-hand-side $\{(x,y) : \widehat{\Delta}(x,y) < t\}$ is $\mathcal{A} \otimes \mathcal{B}$-measurable for every $t \geq 0$
        Hence, the left-hand-side is also measurable for every $t$.
        As a consequence, $\widehat{\Delta}$ is $\mathcal{A} \otimes \mathcal{B}$-measurable.
    \end{proof}

    \subsection{Proof of Proposition \ref{prp:Delta_test_equivalence}}\label{proof:Delta_test_equivalence}
        We start from a measurable $\widehat{\Delta}_\alpha$.
        This induces the tests $\Delta \mapsto \phi_\Delta^\alpha(x)$ through $\phi_\Delta^\alpha(x) = \mathbb{I}\{\Delta \geq \widehat{\Delta}_\alpha(x)\}/\alpha$.
        For every $\Delta \geq 0$, the map $x \mapsto \phi_\Delta^\alpha(x)$ is measurable, since $x \mapsto \mathbb{I}\{\Delta \geq \widehat{\Delta}_\alpha(x)\}$ is measurable by the measurability of $\widehat{\Delta}_\alpha$.
        Moreover, $\Delta \mapsto \phi_\Delta^\alpha(x)$ is non-decreasing and right-continuous for every $x$, by construction.
        Finally, if $\widehat{\Delta}_\alpha$ is valid, then the validity of $\phi_\Delta^\alpha$ for $H_0^\Delta$ immediately follows from \eqref{ineq:guarantee_formal}.
    
        We now consider the reverse direction, starting from a non-decreasing and right-continuous collection of tests $\Delta \mapsto \phi_\Delta^\alpha(x)$ for which the map $x \mapsto \phi_\Delta^\alpha(x)$ is measurable for every $\Delta \geq 0$.
        Define $\widehat{\Delta}_\alpha(x) = \inf\{\Delta : \phi_\Delta^\alpha(x) \geq 1/\alpha\}$ with $\inf \emptyset = \infty$.
        Measurability of $\widehat{\Delta}_\alpha$ follows from Lemma \ref{lem:invert_threshold}.
        For the validity, note that $\Delta \geq \widehat{\Delta}_\alpha(x)$ if and only if $\phi_\Delta^\alpha(x) = 1/\alpha$.
        Hence, if each test $\phi_\Delta^\alpha$ is valid for $H_0^\Delta$, then $\widehat{\Delta}_\alpha$ is valid. \qed

    \subsection{Proof of Theorem \ref{thm:post-hoc_valid}}\label{proof:post-hoc_valid}
        We start from a measurable equivalence curve $(x, \alpha) \mapsto \widehat{\Delta}_\alpha(x)$, and assume that $\alpha \mapsto \widehat{\Delta}_\alpha(x)$ is non-increasing and right-continuous for every $x$.
        This induces a collection $\Delta \mapsto \e_\Delta$ through
        \begin{align*}
            \e_\Delta(x)
                = \sup\left\{\frac{1}{\alpha} : \Delta \geq \widehat{\Delta}_\alpha(x)\right\}
                = \frac{1}{\inf\{\alpha > 0 : \Delta \geq \widehat{\Delta}_\alpha(x)\}},
        \end{align*}
        where we use the conventions $1/\infty = 0$ and $1/0 = \infty$.

        For the measurability, we prepare for applying Lemma \ref{lem:invert_threshold} by defining the p-value $p_\Delta(x) = 1/\e_\Delta(x)$.
        Measurability of $(x, \Delta) \mapsto p_\Delta(x)$ then follows from Lemma \ref{lem:invert_threshold} with $g((x, \Delta), \alpha) := -\widehat{\Delta}_\alpha(x)$ and $h(\Delta) := -\Delta$.
        As a consequence $x \mapsto \e_\Delta(x)$ is measurable.

        We now consider the reverse direction, starting from a non-decreasing and right-continuous collection of e-values $\Delta \mapsto \e_\Delta(x)$ such that for every $\Delta \geq 0$ the map $x \mapsto \e_\Delta(x)$ is measurable.
        Define the equivalence curve through $\widehat{\Delta}_\alpha(x) = \inf\{\Delta \geq 0 : \e_\Delta(x) \geq 1/\alpha\}$.
        Measurability of $(x, \alpha) \mapsto \widehat{\Delta}_\alpha(x)$ then follows from Lemma \ref{lem:invert_threshold} with $g(x, \Delta) = \e_\Delta(x)$ and $h(\alpha) = 1/\alpha$.

        Now for the post-hoc validity: by the equivalence of $\widehat{\Delta}_\alpha \leq \Delta$ and $\e_\Delta \geq 1/\alpha$, the family of tests $\alpha \mapsto \phi_\Delta^\alpha(x) = \mathbb{I}\{\widehat{\Delta}_\alpha(x) \leq \Delta\} / \alpha = \mathbb{I}\{\e_\Delta(x) \geq 1/\alpha\} / \alpha$ is exactly the test family induced by the e-value $\e_\Delta$.
        Therefore, the if-and-only-if follow directly from Theorem 2 of \citet{koning2023post}. \qed

\subsection{Proof of Proposition \ref{prp:mergingcurves}}\label{app:merging}

\begin{proof}[Proof for multiplication of e-value curves]
Define the threshold sets
\[
S_\alpha^\times
:=
\left\{\Delta\ge 0:\e_\Delta^1\e_\Delta^2\ge \frac{1}{\alpha}\right\},
\qquad
S_{\alpha}^j
:=
\left\{\Delta\ge 0:\e_\Delta^j\ge \frac{1}{\alpha}\right\},
\quad j=1,2.
\]
By definition, $
\widehat{\Delta}_\alpha^\times=\inf S_\alpha^\times$ and $
\widehat{\Delta}_{\alpha}^j=\inf S_{\alpha}^j.$ We now first show that
\begin{equation}\label{eq:set_identity_product_merge}
S_\alpha^\times
=
\bigcup_{\substack{\alpha_1,\alpha_2>0\\
\alpha_1\alpha_2=\alpha}}
\bigl(S_{\alpha_1}^1\cap S_{\alpha_2}^2\bigr),
\end{equation}
which we prove by showing both inclusions. The inclusion ``$\supseteq$'' is immediate: if $\Delta\in S_{\alpha_1}^1\cap S_{\alpha_2}^2$ for some $\alpha_1\alpha_2=\alpha$, then $
\e_\Delta^1\ge \frac{1}{\alpha_1}$ and $
\e_\Delta^2\ge \frac{1}{\alpha_2}$ so
$\e_\Delta^1\e_\Delta^2
\ge
\frac{1}{\alpha_1\alpha_2}
=
\frac{1}{\alpha},$
which means $\Delta\in S_\alpha^\times$. For the reverse inclusion, let $\Delta\in S_\alpha^\times$. Then
$\e_\Delta^1\e_\Delta^2\ge 1/\alpha$, so the interval
$\left[1/\e_\Delta^1,\alpha\e_\Delta^2\right]$
is nonempty (with the convention $1/\infty=0$). Choose any $\alpha_1>0$ in this interval and define
$\alpha_2:=\alpha/\alpha_1$. Then $\alpha_1\alpha_2=\alpha$,
$\alpha_1\ge 1/\e_\Delta^1$ implies $\e_\Delta^1\ge 1/\alpha_1$, and
$\alpha_1\le \alpha\e_\Delta^2$ implies
$\alpha_2=\alpha/\alpha_1\ge 1/\e_\Delta^2$, hence $\e_\Delta^2\ge 1/\alpha_2$.
Thus $\Delta\in S_{\alpha_1}^1\cap S_{\alpha_2}^2$, proving \eqref{eq:set_identity_product_merge}.

Now, since each $\Delta\mapsto \e_\Delta^j$ is non-decreasing and right-continuous, $
S_{\alpha}=[\widehat{\Delta}_{\alpha}^j,\infty)$ for $j=1,2.$
Therefore, for every $\alpha_1\alpha_2=\alpha$,
\[
S_{\alpha_1}^1\cap S_{\alpha_2}^2
=
\left[\max\{\widehat{\Delta}_{\alpha_1}^1,\widehat{\Delta}_{\alpha_2}^2\},\infty\right).
\]
Taking the infimum of the union in \eqref{eq:set_identity_product_merge} yields
\[
\widehat{\Delta}_\alpha^\times
=
\inf_{\alpha_1\alpha_2=\alpha}
\max\{\widehat{\Delta}_{\alpha_1}^1,\widehat{\Delta}_{\alpha_2}^2\},
\]
as claimed.
\end{proof}
\begin{proof}[Proof of weighted average of e-value curves]
The proof for the weighted average follows the same structure as the multiplication proof. Define
\[
S_\alpha^w:=\left\{\Delta\ge 0: w\e_\Delta^1+(1-w)\e_\Delta^2\ge 1/\alpha\right\},
\qquad
S_\alpha^j:=\left\{\Delta\ge 0:\e_\Delta^j\ge 1/\alpha\right\},
\quad j=1,2.
\]
By definition, $\widehat{\Delta}_\alpha^w=\inf S_\alpha^w$ and $\widehat{\Delta}_\alpha^j=\inf S_\alpha^j$. We first show
\begin{equation}\label{eq:set_identity_average_merge}
S_\alpha^w
=
\bigcup_{\substack{\alpha_1,\alpha_2>0\\
\frac{w}{\alpha_1}+\frac{1-w}{\alpha_2}= \frac{1}{\alpha}}}
\left(S_{\alpha_1}^1\cap S_{\alpha_2}^2\right).
\end{equation}
For ``$\supseteq$'': if $\Delta\in S_{\alpha_1}^1\cap S_{\alpha_2}^2$ and
$\frac{w}{\alpha_1}+\frac{1-w}{\alpha_2}= \frac{1}{\alpha}$, then
$\e_\Delta^1\ge 1/\alpha_1$ and $\e_\Delta^2\ge 1/\alpha_2$, so
$w\e_\Delta^1+(1-w)\e_\Delta^2
\ge \frac{w}{\alpha_1}+\frac{1-w}{\alpha_2}
= \frac{1}{\alpha}$, hence $\Delta\in S_\alpha^w$.

For ``$\subseteq$'': let $\Delta\in S_\alpha^w$, so
$w\e_\Delta^1+(1-w)\e_\Delta^2\ge 1/\alpha$.
Set
\[
c:=\frac{1/\alpha}{\,w\e_\Delta^1+(1-w)\e_\Delta^2\,}\in(0,1],
\qquad
\alpha_1:=\frac{1}{c\e_\Delta^1},\quad
\alpha_2:=\frac{1}{c\e_\Delta^2}.
\]
Then $\alpha_1,\alpha_2>0$ and $
\e_\Delta^1\ge c\e_\Delta^1=\frac1{\alpha_1}$ and $
\e_\Delta^2\ge c\e_\Delta^2=\frac1{\alpha_2},
$
so $\Delta\in S_{\alpha_1}^1\cap S_{\alpha_2}^2$. Moreover,
$
\frac{w}{\alpha_1}+\frac{1-w}{\alpha_2}
=
wc\e_\Delta^1+(1-w)c\e_\Delta^2
=
c\bigl(w\e_\Delta^1+(1-w)\e_\Delta^2\bigr)
=
\frac1\alpha.$ So \eqref{eq:set_identity_average_merge} holds. Now using the same steps as in the multiplication proof,
\[
\widehat{\Delta}_\alpha^w
=
\inf_{\substack{\alpha_1,\alpha_2>0:\\[0.2em]
\frac{w}{\alpha_1}+\frac{1-w}{\alpha_2}= \frac{1}{\alpha}}}
\max\left\{\widehat{\Delta}_{\alpha_1}^1,\widehat{\Delta}_{\alpha_2}^2\right\},
\]
as claimed.
\end{proof}

\section{Total positivity}\label{app:tp3counterexample}
We first state the definition of STP$_k$ used throughout, following \citet{karlin1968total}. We then construct a family that is STP$_2$ but fails STP$_3$, showing that STP$_2$ is generally insufficient for our $U$-optimal e-variable results.

\subsection{Definition}

\begin{dfn}[TP$_r$ and STP$_r$]\label{def:tp}
Let $\cM$ and $\cX$ be totally ordered sets, and let $\{p_\mu:\mu\in\cM\}$ be densities on $(\cX,\cA)$ with respect to a common dominating measure $\nu$.
For $r\in\mathbb N$, the family is totally positive of order $r$ (TP$_r$) if, for every $n\in\{1,\dots,r\}$, every $\mu_1<\cdots<\mu_n$ in $\cM$, and every $x_1<\cdots<x_n$ in $\cX$,
\[
\det\!\big[p_{\mu_i}(x_j)\big]_{i,j=1}^n \ge 0.
\]
It is strictly totally positive of order $r$ (STP$_r$) if all such determinants are strictly positive.
If the same condition holds for all $n\ge 1$, the family is TP$_\infty$ (or STP$_\infty$ in the strict case).
\end{dfn}

This hierarchy recovers familiar shape constraints: TP$_1$ means non-negativity of the densities, while TP$_2$ is equivalent to the monotone likelihood ratio (MLR) assumption. Throughout the remainder of this section we impose STP$_3$ on $(\mu,x)\mapsto p_\mu(x)$, which is thus a stronger condition than MLR. The assumption is purely order-based and therefore covers both discrete and continuous sample spaces. It is also broad enough for our applications, including one-parameter exponential families and non-central $t$-distributions (which are STP$_\infty$).

\subsection{Counterexample: STP$_2$ does not imply STP$_3$}\label{sec:tp2tp3}

We give a discrete example showing that STP$_2$ alone does not suffice for Theorem~\ref{thm:Uoptimal}. In this case, we focus on the log-utility so the $U$-optimal e-variable is the numeraire given in Corollary \ref{cor:logoptimal}. 

Let $\cX=\{1,2,3\}$ and $\cM=\{1,2,3,4\}$, and define $\{p_\mu:\mu\in\cM\}$ by
\begin{equation}\label{eq:tp2distribution}
A_p
=
\frac{1}{24}
\begin{pmatrix}
16 & 7 & 1 \\
12 & 6 & 6 \\
6 & 6 & 12 \\
2 & 4 & 18
\end{pmatrix},
\end{equation}
where the $(\mu,x)$ entry is $p_\mu(x)$.  
This family is STP$_2$, but it is not STP$_3$ (the determinant of the upper $3\times3$ block is negative).

Now test $H_0:\mu\in\{1,3,4\}$ against $Q: \mu =2$. Consider
\[
\e_c(x):=\frac{p_2(x)}{c\,p_1(x)+(1-c)\,p_3(x)}.
\]
There is a unique $c^*\approx0.558$ such that
\[
\E_{1}[\e_{c^*}(X)]=\E_{3}[\e_{c^*}(X)]=1.
\]
However, $\E_{4}[\e_{c^*}(X)]>1$, so $\e_{c^*}$ is not valid for the composite null $H_0$. Thus STP$_2$ does not guarantee validity (and hence not numeraire-optimality) of the boundary-mixture likelihood ratio form in Corollary \ref{cor:logoptimal}.

\begin{rmk}
For the reduced null $H_0:\mu\in\{1,3\}$ against $Q:\mu=2$, the same $\e_{c^*}$ is valid and has power ($\E_2[\e_{c^*}(X)]>1$). The failure above is therefore specific to the larger composite null.
\end{rmk}

\subsection{Geometric illustration of STP$_2$ and STP$_3$}

Figure~\ref{fig:tp2counterexample} visualizes \eqref{eq:tp2distribution} on the probability simplex.  
Fix $p_1$ and $p_2$, and let $p=(x',y',z')$ be a candidate third pmf.  
The STP$_2$ constraints for $(p_1,p_2,p)$ reduce to
\[
y' > \frac{x'}{2},\qquad z' > y',
\]
which gives the full colored wedge (green plus yellow).

STP$_3$ adds
\[
\det\!\begin{pmatrix}
p_1(1)&p_1(2)&p_1(3)\\
p_2(1)&p_2(2)&p_2(3)\\
p(1)&p(2)&p(3)
\end{pmatrix}>0
\quad\Longleftrightarrow\quad
3x'-7y'+z' > 0.
\]
So the line $\det=0$ through $p_1$ and $p_2$ splits the STP$_2$ wedge into an STP$_3$ side (green) and an STP$_2$-only side (yellow). Geometrically, the STP$_3$ side is the one obtained by moving convexly toward the vertex $(0,0,1)$. In particular, $p_3$ and $p_4$ lie on the STP$_2$-only side, so the family is STP$_2$ but not STP$_3$.

\begin{figure}[H]
    \centering
    \includegraphics[width = 0.8\linewidth]{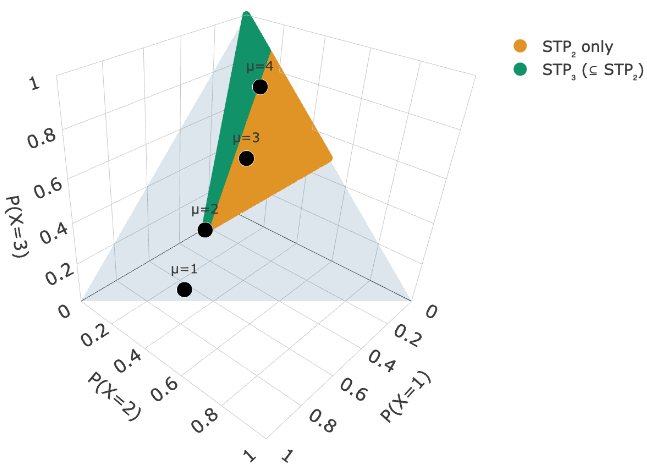}
    \caption{Visualization of the STP$_2$ and STP$_3$ regions in the three-dimensional probability simplex for the example in Section~\ref{sec:tp2tp3}. The gray triangle is the simplex of probability vectors. The yellow region contains points $p$ for which $(p_{\mu_1},p_{\mu_2},p)$ is STP$_2$ but not STP$_3$, while the green region contains points $p$ for which $(p_{\mu_1},p_{\mu_2},p)$ is STP$_3$ (and hence also STP$_2$).} 
    \label{fig:tp2counterexample}
\end{figure}

\section{Proofs for Section \ref{sec:optimalSTP3}}\label{app:proofsopt}
Although the results in the main paper are stated for parameter and sample spaces that are subsets of the real line, in this appendix we study the general two-sided equivalence testing problem on ordered parameter and sample spaces. Let $\cM$ be a totally ordered parameter space, let $(\cX,\mathcal F)$ be a measurable sample space equipped with a total order on $\cX$, and let $\mathcal P:=\{P_\mu:\mu\in\cM\}$ be a statistical model on $(\cX,\mathcal F)$. Assume $\mathcal P$ is dominated by a $\sigma$-finite measure $\nu$ (for instance, Lebesgue or counting measure), and write $p_\mu:=dP_\mu/d\nu$ for $\mu\in\cM$.

Moreover, let $\mu_1<\mu_2$ be in $\cM$, and define $\cM_1:=\{\mu\in\cM:\mu_1<\mu<\mu_2\}$ as the alternative parameter set. We assume $\cM_1\neq\emptyset$. Throughout this appendix, we denote $p_1:=p_{\mu_1}$ and $p_2:=p_{\mu_2}$. Let $w$ be a probability measure on $\cM_1$, and define $q(x):=\int_{\cM_1} p_\mu(x)\,w(d\mu)$.

Lastly, let $\psi:=(U')^{-1}$. For $c\in[0,1]$ and $\lambda>0$, define 
\[
p_c:=c\,p_1+(1-c)\,p_2,\qquad \e_c(x) = \frac{q(x)}{p_c(x)}, \qquad
\e_{c,\lambda}(x):=\psi\!\left(\lambda\,\frac{p_c(x)}{q(x)}\right).
\]

\subsection{Variation diminishing transformations}

For completeness, we recall the notion of variation diminishing transformations used throughout this appendix, following \citet{brown1981variation}.  
For a finite vector $g=(g_1,\dots,g_m)\in\mathbb R^m$, let $S^-(g)$ be the number of sign changes in $(g_1,\dots,g_m)$ after deleting zero entries (with convention $S^-(0,\dots,0)=-1$). Let $S^+(g)$ be the maximal number of sign changes obtainable by replacing each zero entry by either $+1$ or $-1$.

For a function $h:\cX\to\mathbb R$, define for each $V=\{x_1,\ldots,x_m\}\subset\cX$ indexed so that $x_1<\cdots<x_m$, the restriction $
h_V := \big(h(x_1),\dots,h(x_m)\big),$
and set
\[
S^\pm(h):=\sup_{V\subset\cX,\ |V|<\infty} S^\pm(h_V).
\]
Thus all sign-variation statements are understood via finite ordered restrictions.

Then, following \citet{brown1981variation}, we define
\[
h \text{ has at most one local extremum } \iff S^+(h-\gamma)\le 2 \quad \text{for every } \gamma\in\mathbb R.
\]
This is the notion used in Proposition~\ref{prp:unimodality}. Note that it includes both single-peaked and single-dipped shapes, and also monotone functions.

\subsection{Proof of Proposition \ref{prp:unimodality}}

By definition of STP$_3$, for every $\mu\in\cM_1$, the triplet $(p_1,p_\mu,p_2)$ satisfies the STP$_3$ property. In the next lemma, we establish that the same holds for the triplet $(p_1,q,p_2)$.

\begin{lemma}[STP$_3$ is preserved by mixing over $\cM_1$]\label{lem:stp3mixture}
Assume $\{p_\mu:\mu\in\cM\}$ is STP$_3$. Then $(p_1,q,p_2)$ is STP$_3$.
\end{lemma}

\begin{proof}
We verify STP$_1$, STP$_2$, and the $3\times 3$ determinant condition.

STP$_1$ is immediate: $q(x)>0$ since $p_\mu(x)>0$ and $w$ is nonnegative with total mass $1$.

For STP$_2$, fix $x_1<x_2$. Since STP$_3\Rightarrow$ STP$_2$, for each $\mu\in\cM_1$ the ratio
$x\mapsto p_\mu(x)/p_1(x)$ is strictly increasing. Hence
\[
\frac{q(x_2)}{p_1(x_2)}-\frac{q(x_1)}{p_1(x_1)}
=
\int_{\cM_1}
\left(
\frac{p_\mu(x_2)}{p_1(x_2)}-\frac{p_\mu(x_1)}{p_1(x_1)}
\right)w(d\mu)>0,
\]
so $(p_1,q)$ is STP$_2$. Likewise, for each $\mu\in\cM_1$, the ratio
$x\mapsto p_\mu(x)/p_2(x)$ is strictly decreasing, and therefore
\[
\frac{q(x_2)}{p_2(x_2)}-\frac{q(x_1)}{p_2(x_1)}
=
\int_{\cM_1}
\left(
\frac{p_\mu(x_2)}{p_2(x_2)}-\frac{p_\mu(x_1)}{p_2(x_1)}
\right)w(d\mu)<0,
\]
equivalently $p_2/q$ is strictly increasing; thus $(q,p_2)$ is STP$_2$.

Finally, fix $x_1<x_2<x_3$. By linearity of the determinant in the second row,
\[
\det\!\begin{pmatrix}
p_1(x_1)&p_1(x_2)&p_1(x_3)\\
q(x_1)&q(x_2)&q(x_3)\\
p_2(x_1)&p_2(x_2)&p_2(x_3)
\end{pmatrix}
=
\int_{\cM_1}
\det\!\begin{pmatrix}
p_1(x_1)&p_1(x_2)&p_1(x_3)\\
p_{\mu}(x_1)&p_{\mu}(x_2)&p_{\mu}(x_3)\\
p_2(x_1)&p_2(x_2)&p_2(x_3)
\end{pmatrix}
w(d\mu).
\]
Each integrand is strictly positive by STP$_3$ of the original family, so the integral is strictly positive. Therefore $(p_1,q,p_2)$ is STP$_3$.
\end{proof}

\noindent
\large {\textbf{Proposition \ref{prp:unimodality}}}
\normalsize

\begin{proof}
%Fix $c\in(0,1)$ and $\lambda>0$. 
We proceed in two steps. First, we prove that $S^+(\e_c-\gamma)\le 2$ for all $\gamma\in\mathbb R$ and subsequently we show $S^+(\e_{c,\lambda}-\gamma)\le 2$ for all $\gamma\in\mathbb R$ .
Define
\[
h_{c,\gamma}(x):=q(x)-\gamma p_c(x)=p_c(x)\big(\e_c(x)-\gamma\big).
\]
Since $p_c(x)>0$ for all $x$, Proposition 3.1 of \citet{brown1981variation} gives $
S^+(\e_c-\gamma)=S^+(h_{c,\gamma}).$
If $\gamma\le 0$, then $\e_c(x)-\gamma>0$ for all $x$, so $S^+(\e_c-\gamma)=0\le 2$. It remains to show $S^+(h_{c,\gamma})\le 2$ for $\gamma>0$. Fix $\gamma>0$, and let $V=\{x_1,\ldots,x_m\}\subset\cX$ be finite, indexed so that $x_1<\cdots<x_m$. Define
\[
G_V:=
\begin{pmatrix}
p_1(x_1) & q(x_1) & p_2(x_1)\\
\vdots & \vdots & \vdots\\
p_1(x_m) & q(x_m) & p_2(x_m)
\end{pmatrix},
\qquad
b_{c,\gamma}:=(-\gamma c,\,1,\,-\gamma(1-c))^\top.
\]
Then $h_{c,\gamma,V}=G_V b_{c,\gamma}$ and $S^-(b_{c,\gamma})=2$. By Lemma~\ref{lem:stp3mixture}, $(p_1,q,p_2)$ is STP$_3$. Hence, by STP$_3\Leftrightarrow$ SVR$_3$ and transpose symmetry \citep[Theorem 3.2]{brown1981variation},
\[
S^+(G_Vu)\le S^-(u)\quad\text{for all }u\in\mathbb R^3\text{ with }S^-(u)\le 2.
\]
Taking $u=b_{c,\gamma}$ yields $S^+(h_{c,\gamma,V})\le 2$. Since $V$ is arbitrary, $S^+(h_{c,\gamma})\le 2$. Therefore
\[
S^+(\e_c-\gamma)\le 2\quad\text{for all }\gamma\in\mathbb R.
\]

Second, we pass from $\e_c$ to $\e_{c,\lambda}$. Fix $\gamma\in\mathbb R$ and define $h_\gamma(x):=\e_{c,\lambda}(x)-\gamma$. If $\gamma\notin \psi((0,\infty))$, then $h_\gamma$ has constant sign and $S^+(h_\gamma)=0$.
If $\gamma\in \psi((0,\infty))$, set $
t_\gamma:=\frac{\lambda}{\psi^{-1}(\gamma)}>0.$
Since $\psi$ is strictly decreasing and $\e_c(x)=q(x)/p_c(x)>0$,
\[
\operatorname{sgn}(h_\gamma(x))
=\operatorname{sgn}\!\left(\psi\!\left(\lambda\,\frac{p_c(x)}{q(x)}\right)-\gamma\right)
=\operatorname{sgn}\!\left(\e_c(x)-t_\gamma\right).
\]
Hence
\[
S^+(h_\gamma)=S^+(\e_c-t_\gamma)\le 2
\]
by the first step. So $S^+(\e_{c,\lambda}-\gamma)\le 2$ for all $\gamma\in\mathbb R$.
\end{proof}

\subsection{Proof of Proposition \ref{prp:determineconstantutility}}

\begin{proof}
We prove the result in two parts: first we identify $c$ for fixed $\lambda$, then we identify $\lambda$.

Fix $\lambda>0$, and define
\[
A_\lambda(c):=\E_{\mu_1}[\e_{c,\lambda}(X)],\qquad
B_\lambda(c):=\E_{\mu_2}[\e_{c,\lambda}(X)],\qquad
g_\lambda(c):=A_\lambda(c)-B_\lambda(c).
\]
Let $
C_U:=\sup_{x>0} xU'(x)<\infty.$
We then obtain that for any $t>0$ with $x=\psi(t)$ (so $t=U'(x)$), which means $t\,\psi(t)=xU'(x)\le C_U$ and  
hence $
\psi(t)\le \frac{C_U}{t}$.

We now show that: (i) $g_\lambda$ is strictly decreasing on $(0,1)$; (ii) $g_\lambda$ is continuous on $(0,1)$; (iii) $g_\lambda(c)>0$ for $c$ near $0$ and $g_\lambda(c)<0$ for $c$ near $1$. 

For (i), let $0<c<c'<1$. Since
\[
\frac{p_c-p_{c'}}{q}=(c-c')\frac{p_1-p_2}{q},
\]
and $\psi$ is strictly decreasing, we get that $
\operatorname{sgn}\!\big(\e_{c,\lambda}-\e_{c',\lambda}\big)
=
\operatorname{sgn}(p_1-p_2).$
Therefore
\[
g_\lambda(c)-g_\lambda(c')
=
\int_{\cX}\big(\e_{c,\lambda}-\e_{c',\lambda}\big)(p_1-p_2)\,d\nu>0,
\]
so $g_\lambda$ is strictly decreasing in $c$.

For (ii), fix $c_0\in(0,1)$ and let $c_n\to c_0$. Choose $\delta\in(0,\min\{c_0,1-c_0\})$. For all large $n$, we then know that $c_n\in[\delta,1-\delta]$. Set $m_\delta:=\min\{\delta,1-\delta\}>0$ which gives us that $
p_{c_n}\ge m_\delta(p_1+p_2),$
so, by $\psi(t)\le C_U/t$,
\[
\e_{c_n,\lambda}
=
\psi\!\left(\lambda\frac{p_{c_n}}{q}\right)
\le
\frac{C_U}{\lambda}\frac{q}{p_{c_n}}
\le
\frac{C_U}{\lambda m_\delta}\frac{q}{p_1+p_2}.
\]
Hence
\[
\big|\e_{c_n,\lambda}(p_1-p_2)\big|
\le
\frac{C_U}{\lambda m_\delta}\,q,
\]
and the right-hand side is integrable since $\int_{\cX} q\,d\nu=1$. Also, we have pointwise convergence $
\e_{c_n,\lambda}(x)\to \e_{c_0,\lambda}(x)$
because $c\mapsto p_c(x)$ is affine and $\psi$ is continuous on $(0,\infty)$. Now by applying the dominated convergence theorem, we know that $
g_\lambda(c_n)\to g_\lambda(c_0),$
so $g_\lambda$ is continuous on $(0,1)$.

For (iii), define
\[
h_0(x):=\psi\!\left(\lambda\frac{p_2(x)}{q(x)}\right),\qquad
h_1(x):=\psi\!\left(\lambda\frac{p_1(x)}{q(x)}\right),
\]
which will be helpful to study the behavior of $g_{\lambda}(c)$ for $c \downarrow 0$ and $c \uparrow 1$.
By Lemma~\ref{lem:stp3mixture}, $(p_1,q,p_2)$ is STP$_3$, hence pairwise STP$_2$. Thus $p_2/q$ is increasing and $p_1/q$ is decreasing. Since $\psi$ is strictly decreasing, $h_0$ is decreasing and $h_1$ is increasing. By strict SVR$_2$/MLR ordering,
\[
A_0:=\E_{\mu_1}[h_0]>\E_{\mu_2}[h_0]=:B_0,\qquad
A_1:=\E_{\mu_1}[h_1]<\E_{\mu_2}[h_1]=:B_1.
\]
Moreover, using $\psi(t)\le C_U/t$,
\[
B_0=\int h_0\,p_2\,d\nu
\le \frac{C_U}{\lambda}\int q\,d\nu
=\frac{C_U}{\lambda}<\infty,
\]
and similarly $A_1<\infty$. Now let $c_n\downarrow0$ which gives pointwise convergence $\e_{c_n,\lambda}\to h_0$. Then by applying Fatou's lemma $
\liminf_{n\to\infty}A_\lambda(c_n)\ge A_0.$
Also, for $n$ large enough (that is, $c_n\le 1/2$), we get the following bound $p_2\e_{c_n,\lambda}
\le
\frac{C_U}{\lambda}\frac{q\,p_2}{p_{c_n}}
\le
\frac{2C_U}{\lambda}q$ which is integrable.
Thus, by dominated convergence, $
B_\lambda(c_n)\to B_0.$
Hence
\[
\liminf_{n\to\infty} g_\lambda(c_n)\ge A_0-B_0>0.
\]
Since $(c_n)$ was arbitrary, $\liminf_{c\downarrow0} g_\lambda(c)>0$.

Similarly, let $c_n\uparrow1$. %Then pointwise $\e_{c_n,\lambda}\to h_1$, and for $n$ large enough ($c_n\ge 1/2$),
%\[\e_{c_n,\lambda}p_1\le\frac{C_U}{\lambda}\frac{q\,p_1}{p_{c_n}}\le\frac{2C_U}{\lambda}q,\] 
%so $A_\lambda(c_n)\to A_1$ by dominated convergence. 
Using analogous arguments, we get by Fatou that $
\liminf_{n\to\infty}B_\lambda(c_n)\ge B_1$
which yields $\limsup_{c\uparrow1} g_\lambda(c)<0$. Thus there exist $c_-,c_+\in(0,1)$ with $c_-<c_+$ such that $
g_\lambda(c_-)>0>g_\lambda(c_+).$
By continuity of $g_\lambda$ on $(0,1)$, IVT yields $c(\lambda)\in(c_-,c_+)$ with
\[
A_\lambda(c(\lambda))=B_\lambda(c(\lambda)).
\]
By strict monotonicity, this $c(\lambda)$ is unique. 

Let us now identify $\lambda$. Define
\[
m(\lambda):=A_\lambda(c(\lambda))=B_\lambda(c(\lambda)).
\]
We note that this function depends on $\lambda$ via the $\lambda$ in $A_{\lambda}$ and $B_{\lambda}$, but also because it pinpoints $c(\lambda)$. We now show that this function is continuous in $\lambda$ after which we will apply the IVT. Fix $\lambda_0>0$, and let $\lambda_n\to\lambda_0$.
Write $c_n:=c(\lambda_n)$, $c_0:=c(\lambda_0)$.

First, $g(c,\lambda):=\int_{\cX}\psi\!\left(\lambda\frac{p_c}{q}\right)(p_1-p_2)\,d\nu$
is jointly continuous on $(0,1)\times(0,\infty)$: for $(c,\lambda)$ in a small rectangle
$[c_-,c_+]\times[\underline\lambda,\bar\lambda]$ around $(c_0,\lambda_0)$, we have
$p_c\ge m(p_1+p_2)$ for some $m>0$, $\underline\lambda>0$, and
\[
\left|\psi\!\left(\lambda\frac{p_c}{q}\right)(p_1-p_2)\right|
\le \frac{C_U}{\underline\lambda}\frac{q}{p_c}|p_1-p_2|
\le \frac{C_U}{\underline\lambda\,m}\,q,
\]
with $\int q\,d\nu=1$. Pointwise continuity then gives joint continuity by DCT.

Second, $\lambda \mapsto c(\lambda)$ is continuous on $(0,\infty)$: since $c\mapsto g_\lambda(c)$ is strictly decreasing and $g_{\lambda_0}(c_0)=0$, choose
$a<c_0<b$ such that $g_{\lambda_0}(a)>0>g_{\lambda_0}(b)$. By joint continuity, for all
$\lambda$ near $\lambda_0$, $g_\lambda(a)>0>g_\lambda(b)$, which means that $c(\lambda)\in(a,b)$.
Now take any subsequence $c_{n_k}$; it has a further subsequence $c_{n_{k_j}}$ converging to some
$\bar c\in[a,b]$ by sequential compactness of $[a,b]$. Because $g_{\lambda_{n_{k_j}}}(c_{n_{k_j}})=0$ and $g$ is jointly continuous,
$g_{\lambda_0}(\bar c)=0$. Uniqueness of the zero implies $\bar c=c_0$. Therefore
$c_n\to c_0$, i.e. $\lambda\mapsto c(\lambda)$ is continuous.

Finally, using that $A_\lambda(c(\lambda))=B_\lambda(c(\lambda))$ we can also write
\[
m(\lambda_n)=\int_{\cX}\psi\!\left(\lambda_n\frac{p_{c(\lambda_n)}}{q}\right)p_{c(\lambda_n)}\,d\nu.
\]
Pointwise, $p_{c(\lambda_n)}(x)\to p_{c(\lambda_0)}(x)$ by continuity of $\lambda \mapsto c(\lambda)$, so the integrand converges to
$\psi\!\left(\lambda_0\frac{p_{c_0}}{q}\right)p_{c_0}$.
For $n$ large, $\lambda_n\ge \lambda_0/2$, we note $
0\le \psi\!\left(\lambda_n\frac{p_{c_n}}{q}\right)p_{c_n}
\le \frac{2C_U}{\lambda_0}q,$
which is integrable. By applying the DCT, $
m(\lambda_n)\to m(\lambda_0).$ This yields that $\lambda\mapsto m(\lambda)$ is continuous on $(0,\infty)$.

To now prove the existence of $\lambda^*$, we apply the IVT. Since $m$ is continuous on $(0,\infty)$, it suffices to show $
\lim_{\lambda\downarrow 0} m(\lambda)=\infty$ and
$
\lim_{\lambda\uparrow\infty} m(\lambda)=0.
$
For the upper tail, we note that
\[
m(\lambda)
=
\E_{p_{c(\lambda)}}\!\left[\psi\!\left(\lambda\frac{p_{c(\lambda)}}{q}\right)\right]
\le
\frac{C_U}{\lambda}\E_{p_{c(\lambda)}}\!\left[\frac{q}{p_{c(\lambda)}}\right]
=
\frac{C_U}{\lambda}\int q\,d\nu
=
\frac{C_U}{\lambda}
\to 0
\quad(\lambda\uparrow\infty).
\]

For the lower tail, let $\lambda_n\downarrow 0$, and we again denote $c_n:=c(\lambda_n)\in(0,1)$.
By compactness of $[0,1]$, pass to a subsequence (not relabeled) with $c_n\to\bar c\in[0,1]$.
Then $p_{c_n}(x)\to p_{\bar c}(x)$, and for each $x$, $
\lambda_n\frac{p_{c_n}(x)}{q(x)}\to 0$ which implies that $
\psi\!\left(\lambda_n\frac{p_{c_n}(x)}{q(x)}\right)\to\infty.$
Thus, by Fatou's lemma,
\[
\liminf_{n\to\infty} m(\lambda_n)
=
\liminf_{n\to\infty}\int \psi\!\left(\lambda_n\frac{p_{c_n}}{q}\right)p_{c_n}\,d\nu
\ge
\int \liminf_{n\to\infty}\psi\!\left(\lambda_n\frac{p_{c_n}}{q}\right)p_{c_n}\,d\nu
=
\infty,
\]
so we get $\lim_{\lambda\downarrow 0} m(\lambda)=\infty$. Therefore $m$ takes values above and below $1$, and by the intermediate value theorem there exists $\lambda^*>0$ such that $
m(\lambda^*)=1.$
By definition of $m$, this yields
\[
\E_{\mu_1}[\e_{c(\lambda^*),\lambda^*}(X)]
=
\E_{\mu_2}[\e_{c(\lambda^*),\lambda^*}(X)]
=
1.
\]
Setting $c^*:=c(\lambda^*)$ proves existence of $(c^*,\lambda^*)$.
\end{proof}

\subsection{Proof of Theorem \ref{thm:Uoptimal}}\label{app:Uoptimalresult}

The following lemma generalizes Theorem 4.1 of \cite{larsson2025numeraire}. It provides a verification result that simplifies checking whether a candidate pair $(\e^*,P^*)$ is optimal in the expected-utility sense. We assume that $U$ satisfies the same conditions as in the paper. %increasing, concave, continuously differentiable, and satisfies the Inada-type conditions $U(0)=\lim_{x\downarrow 0}U(x)=0,$ $U(\infty)=\lim_{x\uparrow\infty}U(x)=\infty,$ $U'(0)=\lim_{x\downarrow 0}U'(x)=\infty,$ $U'(\infty)=\lim_{x\uparrow\infty}U'(x)=0.$

\begin{lemma}\label{lem:utility4.1Larsson}
Assume that $Q \ll \cP$. Let $\e^*\in\cE$ be $Q$-a.s. strictly positive, and assume $\lambda:=\E_Q[U'(\e^*)\e^*]\in(0,\infty)$. Define $P^*$ by $dP^*/dQ=U'(\e^*)/\lambda$. Then $\e^*$ is $U$-optimal if and only if $P^*\in\cP_{\mathrm{eff}}$.
\end{lemma}

\begin{proof}
By Theorem 2.5 of \citet{larsson2025numeraire}, all e-variables are $Q$-a.s. finite under $Q\ll\cP$. Since $\e^*>0$ $Q$-a.s. and $U'>0$ on $(0,\infty)$, the Radon--Nikodym derivative $dP^*/dQ=U'(\e^*)/\lambda$ is well-defined $Q$-a.s.

For the implication $(\Rightarrow)$, assume $\e^*$ is $U$-optimal. Then Theorem 3 of \citet{koning2024continuous} gives the first-order condition
\[
\E_Q\!\left[U'(\e^*)(\e-\e^*)\right]\le 0,\qquad \forall \e\in\cE.
\]
Rearranging yields $\E_Q[U'(\e^*)\e]\le \E_Q[U'(\e^*)\e^*]=\lambda$. Dividing by $\lambda>0$, we get $
\E_Q\!\left[\e\,U'(\e^*)/\lambda\right]\le 1,$  for all $ \e\in\cE.$
By definition of $P^*$, this is exactly $\E_{P^*}[\e]\le 1$ for all $\e\in\cE$, so $P^*\in\cP_{\mathrm{eff}}$.

For $(\Leftarrow)$, assume $P^*\in\cP_{\mathrm{eff}}$. By definition of the effective null, $\E_{P^*}[\e]\le 1$ for every $\e\in\cE$. Multiply by $\lambda$ and use $dP^*/dQ=U'(\e^*)/\lambda$:
\[
\lambda\big(\E_{P^*}[\e]-1\big)
=
\E_Q\!\left[\e\,U'(\e^*)\right]-\E_Q\!\left[\e^*U'(\e^*)\right]
=
\E_Q\!\left[U'(\e^*)(\e-\e^*)\right]
\le 0.
\]
Hence the same first-order condition holds for all $\e\in\cE$. Applying Theorem 3 of \citet{koning2024continuous} again, we conclude that $\e^*$ is $U$-optimal.
\end{proof}

\noindent
\large {\textbf{Theorem} \ref{thm:Uoptimal}}
\normalsize

\begin{proof}
Let
\[
I(\mu):=\E_\mu[\e_{c^*,\lambda^*}(X)],\qquad \mu\in\cM,
\]
where $(c^*,\lambda^*)$ is chosen as in Proposition~\ref{prp:determineconstantutility}, so $I(\mu_1)=I(\mu_2)=1$.

For $\gamma\in\mathbb R$, set $h_\gamma(x):=\e_{c^*,\lambda^*}(x)-\gamma$. By Proposition \ref{prp:unimodality}, we have that $S^+(h_\gamma)\le 2$. Since $\{p_\mu\}$ is SVR$_3$, we immediately obtain that $S^+(I-\gamma)\le 2$ for all $\gamma \in \mathbb{R}$, because $I(\mu)-\gamma=\int h_\gamma(x)p_\mu(x)\,d\nu(x)$. Define $J(\mu):=I(\mu)-1$. Then $S^+(J)\le 2$ and $J(\mu_1)=J(\mu_2)=0$. So, only two sign configurations are possible: either $J>0$ on $\cM_1$ and $J<0$ on $\cM_0$, or the reverse.

We now exclude the reverse pattern. Write $R:=q/p_{c^*}$, and let $P^*$ be the measure associated with density $p_{c^*}$. Then $\E_Q[\e_{c^*,\lambda^*}(X)]=\E_{P^*}[R\,\psi(\lambda^*/R)]$. Because $\psi$ is strictly decreasing, the map $r\mapsto \psi(\lambda^*/r)$ is strictly increasing. Therefore both $r\mapsto r$ and $r\mapsto \psi(\lambda^*/r)$ are increasing functions of the same variable, so by the covariance inequality we obtain
$\E_{P^*}[R\,\psi(\lambda^*/R)] \ge \E_{P^*}[R]\E_{P^*}[\psi(\lambda^*/R)]$.
Now $\E_{P^*}[R]=\int q\,d\nu=1$, and
$\E_{P^*}[\psi(\lambda^*/R)]=\E_{P^*}[\e_{c^*,\lambda^*}]
=c^*\E_{\mu_1}[\e_{c^*,\lambda^*}]+(1-c^*)\E_{\mu_2}[\e_{c^*,\lambda^*}]=1$,
where the last equality uses Proposition~\ref{prp:determineconstantutility}. This means that \(\E_Q[\e_{c^*,\lambda^*}(X)]\ge 1\).

Suppose, for contradiction, that the reverse sign pattern holds, i.e. \(J(\mu)<0\) for all \(\mu\in\cM_1\). Then \(I(\mu)=1+J(\mu)<1\) on \(\cM_1\). Since \(Q\) is a mixture over \(\mu\in\cM_1\), we get
\[
\E_Q[\e_{c^*,\lambda^*}]
=
\int_{\cM_1}\E_\mu[\e_{c^*,\lambda^*}]\,w(d\mu)
=
\int_{\cM_1} I(\mu)\,w(d\mu)
<1,
\]
which contradicts \(\E_Q[\e_{c^*,\lambda^*}(X)]\ge 1\). So the reverse pattern is impossible. Therefore \(J(\mu)\le 0\) on \(\cM_0\), i.e. \(I(\mu)\le 1\) for \(\mu\in\cM_0\). This proves that \(\e_{c^*,\lambda^*}\) is a valid e-variable for the null.

To conclude optimality, note that $U'(\e_{c^*,\lambda^*})=\lambda^*\,dP^*/dQ$ by which we can obtain that $dP^*/dQ=U'(\e_{c^*,\lambda^*})/\E_Q[U'(\e_{c^*,\lambda^*})\,\e_{c^*,\lambda^*}]$. Next to that, \(P^*\in\cP_{\mathrm{eff}}\) because for every \(\e\in\cE\), \(\E_{P^*}[\e]=c^*\E_{\mu_1}[\e]+(1-c^*)\E_{\mu_2}[\e]\le 1\) where $\lambda^* < \infty$. Now, Lemma~\ref{lem:utility4.1Larsson} implies that \(\e_{c^*,\lambda^*}\) is \(U\)-optimal. \end{proof}

\section{Proofs for Section \ref{sec:seqtesting}}

\subsection{Proof of Proposition \ref{prp:test_martingale}}\label{app:proofprptestmart}

Note that the right-hand-side is equivalent to $\E^P[\e^{t+1} \mid \mathcal{F}_{t}] \leq \e^{t}$ a.s for every $P \in H$.
The right-to-left implication then follows from the optional sampling theorem (see e.g. Proposition 1.16 in the Supplementary Material of \citet{wang2025extended}).
The left-to-right implication follows from the fact that $\tau = t$ and $\sigma = t - 1$ are stopping times, and defining the conditional e-values through $\e_s = \e^{s} / \e^{s-1}$ if $\e^{s-1} > 0$ and $\e_s = 1$ if $\e^{s-1} = 0$. \qed

\section{Simulations results}

\subsection{Comparison of sequential TOST-E and symmetric $t$-squared}\label{app:comparetostesymt}

We first compare $\{\e_{\mathrm{TOST}}^n\}$ and $\{\e_{\Delta_S}^n\}$ in the symmetric effect-size setting with margins $\pm 0.5$, i.e.,
$H_0:\delta\le -0.5\ \text{or}\ \delta\ge 0.5$
versus
$H_1:\delta=0$.
Data are generated as $X_1,\ldots,X_n\stackrel{iid}{\sim}\mathcal N(0,1)$, so the alternative is true. For each Monte Carlo repetition, both e-processes are computed for all $2\le n\le 50$, with $M=50{,}000$ repetitions.

We view e-values primarily as continuous measures of evidence, and recommend using them in that way \citep{koning2024continuous}. Accordingly, we first report the Monte Carlo average e-value as a function of $n$ (left pane of Figure~\ref{fig:final_figure}). For comparability with standard power summaries, we additionally report the sequential rejection probability
$\mathbb P_{0}\!\left(\sup_{t\le n}\e^t\ge 1/\alpha\right)$ at $\alpha=0.05$, justified by Ville’s inequality (right pane of Figure~\ref{fig:final_figure}).
The symmetric $t$-squared (LR-based) process dominates TOST-E throughout $n$: average e-values are up to 54\% larger, and rejection probabilities are up to 16\% higher. This serves as a benchmark for the asymmetric comparison below.

\captionsetup{font=small}
\captionsetup[subfigure]{font=small,labelformat=empty,justification=centering}

% Preamble

% Match your ggplot colors
\definecolor{symLR}{HTML}{E69F00} % Symmetric LR
\definecolor{tost}{HTML}{0072B2}  % TOST

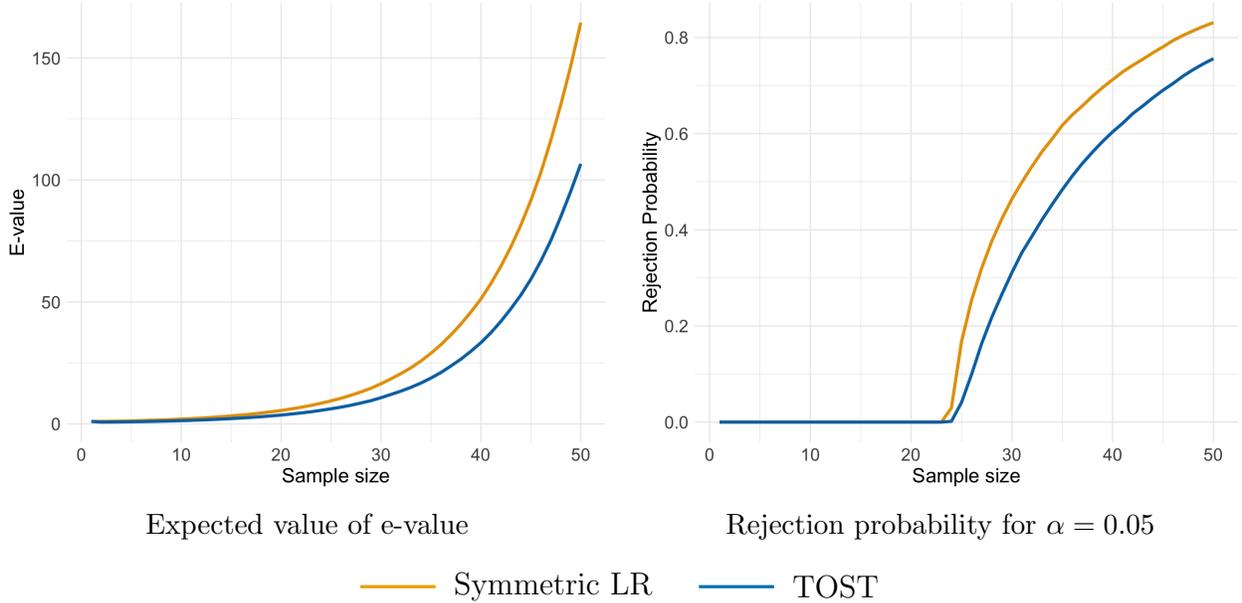
\begin{figure}[htbp]\label{fig:comparetostsym}
  \centering
  \begingroup
  \captionsetup[subfigure]{labelformat=empty}
  
  % --- Two plots side by side ---
  \begin{subfigure}[t]{0.49\textwidth}
    \centering
    \includegraphics[width=\linewidth]{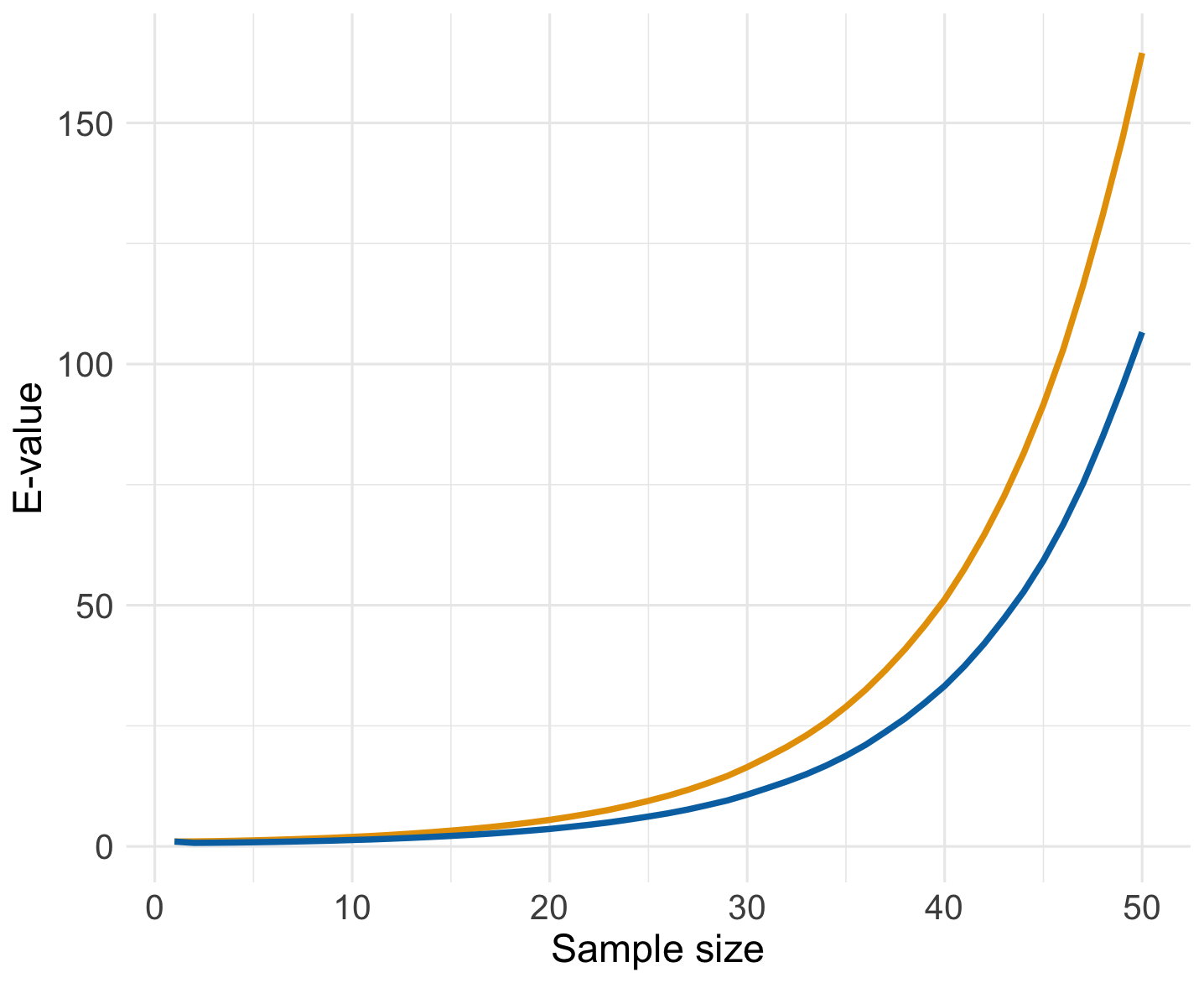}
    \caption{Expected value of e-value}
    \label{fig:eval}
  \end{subfigure}
  \hfill
  \begin{subfigure}[t]{0.49\textwidth}
    \centering
    \includegraphics[width=\linewidth]{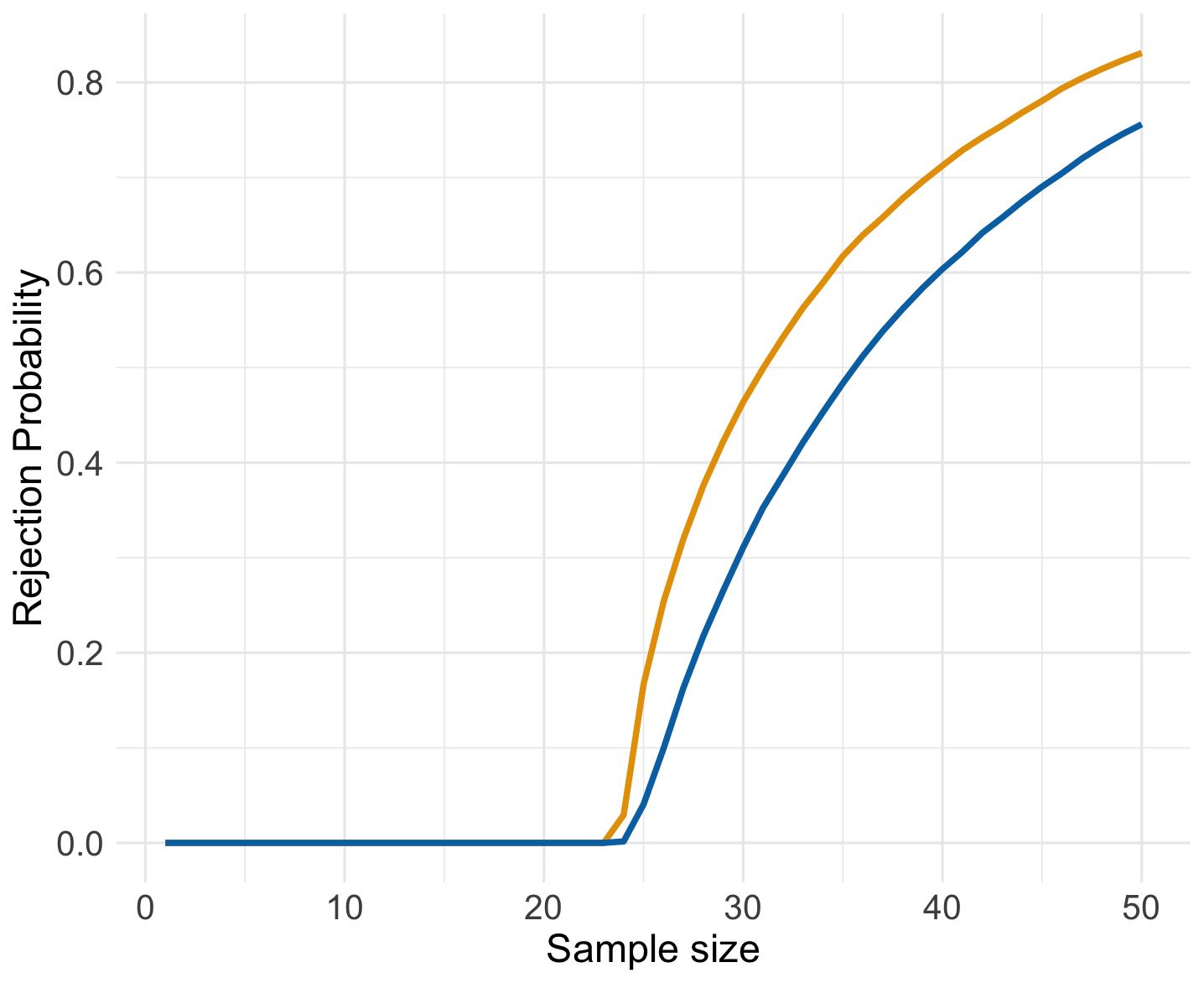}
    \caption{Rejection probability for $\alpha = 0.05$}
    \label{fig:rej}
  \end{subfigure}
  \endgroup

  \vspace{0.5em}

  % --- Compact legend on a single line ---
  \begin{tikzpicture}[baseline]
    % Symmetric LR
    \draw[very thick, symLR] (0,0) -- (1.0,0);
    \node[anchor=west] at (1.1,0) {Symmetric LR};
    
    % Space between legend items
    \node at (4.0,0) {}; % adjust as needed

    % TOST
    \draw[very thick, tost] (4.5,0) -- (5.5,0);
    \node[anchor=west] at (5.6,0) {TOST};
  \end{tikzpicture}
\caption{Comparison of sequential TOST-E and symmetric $t$-squared test under the alternative $\delta=0$, with $X_1,\ldots,X_n\sim\mathcal N(0,1)$, sample sizes $2\le n\le 50$, and $M=50{,}000$ replications.}
  \label{fig:final_figure}
\end{figure}\vspace{-0.2cm}

\subsection{Comparison of sequential TOST-E and product of numeraires}\label{app:comparetostenum}

We next compare $\{\e_{\mathrm{TOST}}^n\}$ with the product-of-numeraires process $\{\e_{\mathrm{num}}^n\}$ for asymmetric margins $(\Delta^-,\Delta^+)=(-0.6,0.4)$, testing
$H_0:\mu\le \Delta^- \ \text{or}\ \mu\ge \Delta^+$
against
$H_1:\Delta^-<\mu<\Delta^+$,
under $X_1,\ldots,X_n \stackrel{iid}{\sim}\mathcal N(\mu,1)$.
We consider two interior alternatives, $\mu=0$ and $\mu=0.3$, and compute both e-processes for $2\le n\le 75$ with $M=50{,}000$ repetitions.

Figure~\ref{fig:comp_num_tost_both} summarizes the results (top panel: $\mu=0$ and bottom panel: $\mu=0.3$). In both DGPs, TOST-E attains larger mean e-values and larger sequential rejection probabilities
$\mathbb P_{\mu}\!\left(\sup_{t\le n}\e^t\ge 1/\alpha\right)$ at $\alpha=0.05$.
The gap is much larger at $\mu=0$ and smaller at $\mu=0.3$, which is closer to the upper margin $\Delta^+=0.4$.
A useful intuition is that the product-of-numeraires construction uses one boundary-mixture denominator that must hedge both null boundaries at each step; in asymmetric settings this hedge can be misaligned with interior alternatives, diluting evidence growth relative to one-sided TOST components.

% If not already loaded:
% \usepackage{subcaption}
% \usepackage{tikz}
% \usepackage{xcolor}

\definecolor{numproc}{HTML}{E69F00} % Product of numeraires
\definecolor{tost}{HTML}{0072B2}    % TOST-E

\begin{figure}[htbp]
  \centering
  \begingroup
  \captionsetup[subfigure]{labelformat=empty,textfont=bf}

  \begin{subfigure}[t]{0.83\textwidth}
    \centering
    \includegraphics[width=\linewidth]{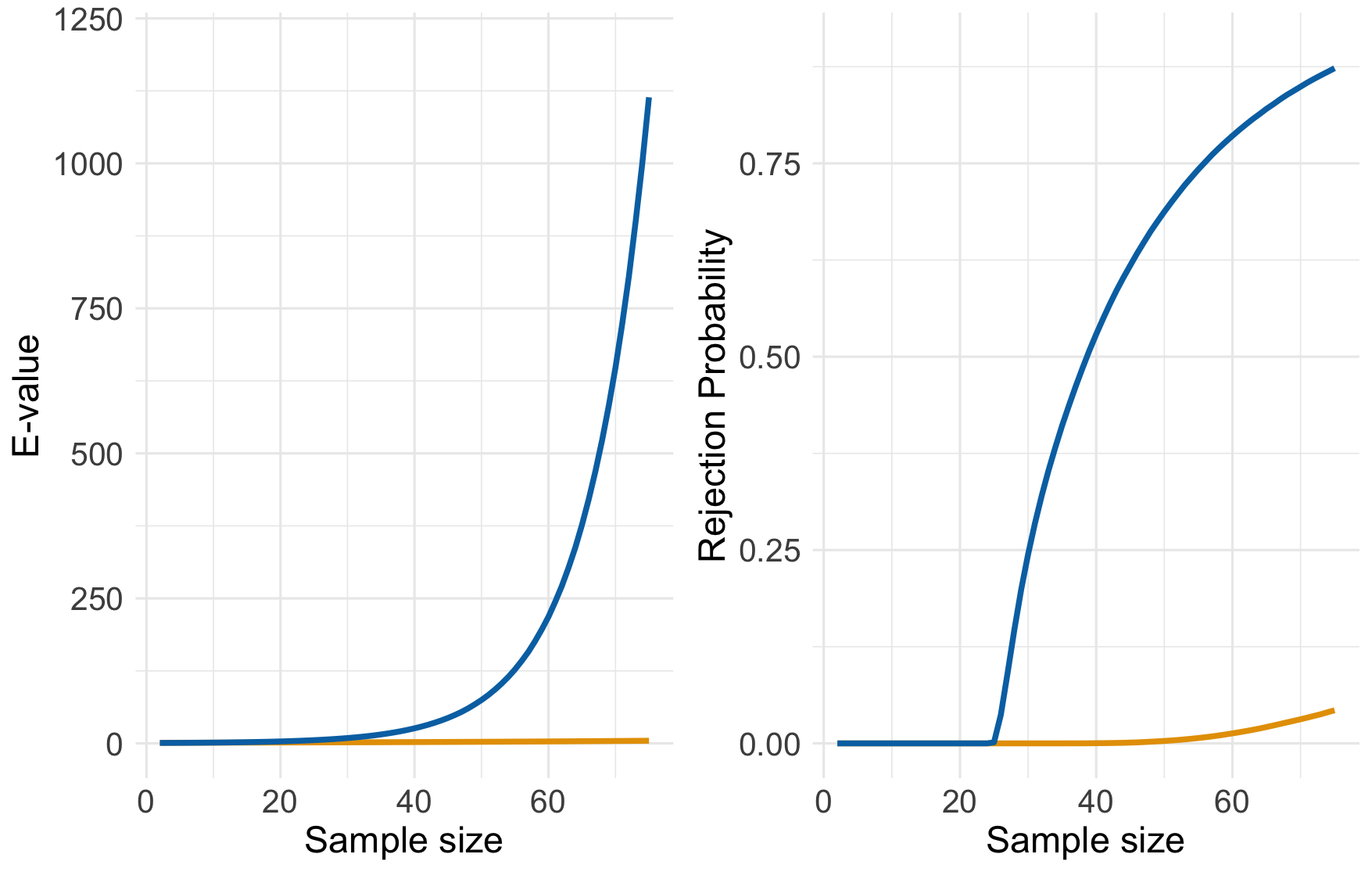}
    \vspace{-0.2em}

    \begin{minipage}[t]{0.49\linewidth}\centering
       {\small Expected value of e-value\hspace{-0.2cm}}
    \end{minipage}\hfill
    \begin{minipage}[t]{0.49\linewidth}\centering
       {\small\quad\quad Rejection probability for $\alpha=0.05$}
    \end{minipage}
    \caption{\bm{$\mu=0$}}
    \label{fig:comp_num_tost_mu0}
  \end{subfigure}

  \vspace{0.3em}

  \begin{subfigure}[t]{0.83\textwidth}
    \centering
    \includegraphics[width=\linewidth]{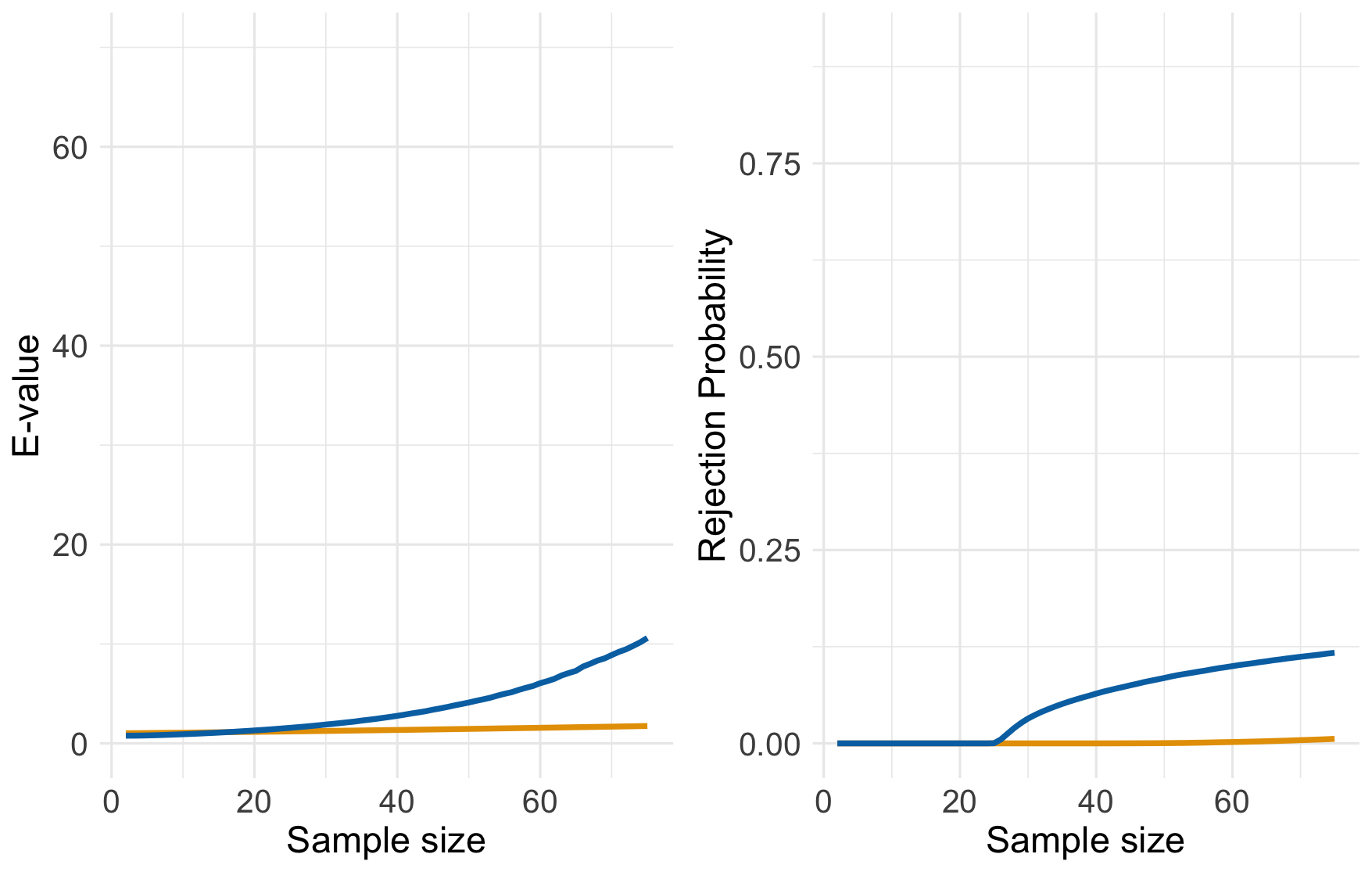}
    \vspace{-0.2em}

    \begin{minipage}[t]{0.49\linewidth}\centering
      {\small Expected value of e-value\hspace{-0.2cm}}
    \end{minipage}\hfill
    \begin{minipage}[t]{0.49\linewidth}\centering
      {\small \quad\quad Rejection probability for $\alpha=0.05$}
    \end{minipage}

    \caption{\bm{$\mu=0.3$}}
    \label{fig:comp_num_tost_mu03}
  \end{subfigure}
  \endgroup

  \vspace{0.5em}
  \begin{tikzpicture}[baseline]
    \draw[very thick, numproc] (0,0) -- (1.0,0);
    \node[anchor=west] at (1.1,0) {Product of numeraires};
    \draw[very thick, tost] (5.4,0) -- (6.4,0);
    \node[anchor=west] at (6.5,0) {TOST};
  \end{tikzpicture}

  \caption{Comparison of sequential TOST-E and product of numeraires for asymmetric margins $(\Delta^-,\Delta^+)=(-0.6,0.4)$ under $X_1,\ldots,X_n\sim\mathcal N(\mu,1)$, with $2\le n\le 75$ and $M=200{,}000$ replications.}
  \label{fig:comp_num_tost_both}
\end{figure}

%\renewcommand{\thesubsection}{\Alph{subsection}}
%\appendix
%\section*{Appendix}

%\input{Appendix}

\end{document}